\documentclass{amsart}
\usepackage[margin=1in]{geometry}
\usepackage{mathrsfs,nicefrac,graphicx,float,cancel,tikz,esdiff,mathtools,bm}
\usepackage{lipsum}
\usepackage{float,graphicx,subcaption}
\usepackage{enumitem,hyperref,nameref}
\usepackage{subfiles}
\usepackage{bigints}
\usepackage[toc,page]{appendix}
\usepackage[utf8]{inputenc}

\usepackage{bigints}

\theoremstyle{plain}
\newtheorem{theorem}{Theorem}
\numberwithin{theorem}{section}

\newtheorem{lemma}[theorem]{Lemma}
\newtheorem{prop}[theorem]{Proposition}

\newtheorem{corollary}[theorem]{Corollary}

\makeatletter
\newtheorem*{rep@theorem}{\rep@title}
\newcommand{\newreptheorem}[2]{%
\newenvironment{rep#1}[1]{%
 \def\rep@title{#2 \ref{##1}}%
 \begin{rep@theorem}}%
 {\end{rep@theorem}}}
\makeatother

\newreptheorem{theorem}{Theorem}
\newreptheorem{lemma}{Lemma}
\newreptheorem{prop}{Proposition}

\DeclareMathOperator{\arcsinh}{arcsinh}

\DeclareMathOperator{\sech}{sech}
\DeclareMathOperator{\sgn}{sgn}
\DeclareMathOperator{\sinc}{sinc}

\DeclareMathOperator{\Rez}{Re}
\DeclareMathOperator{\Imz}{Im}

\theoremstyle{definition}
\newtheorem{definition}[theorem]{Definition}

\theoremstyle{remark}
\newtheorem*{remark}{Remark}

\def\thmheadbrackets#1#2#3{%
	\thmname{#1}\thmnumber{\@ifnotempty{#1}{ }\@upn{#2}}%
	\thmnote{ {\the\thm@notefont[#3]}}}
\makeatother

\newtheoremstyle{brakets}
{}
{}
{\itshape}
{}
{\bfseries}
{.}
{ }
{\thmheadbrackets{#1}{#2}{#3}}

\theoremstyle{brakets}

\newcommand{\R}{\mathbb{R}}
\newcommand{\Z}{\mathbb{Z}}
\newcommand{\C}{\mathbb{C}}
\newcommand{\N}{\mathbb{N}}

\newcommand{\vep}{\varepsilon}

\title{Complex Methods in the Asymptotics of M\"obius energy integrals of helix curves}
\author{Max Lipton}

\newcommand{\Addresses}{{
  \bigskip

  Max Lipton, \textsc{Department of Mathematics, Massachusetts Institute of Technology}\par\nopagebreak
  \textit{E-mail address}: \texttt{liptonm@mit.edu}

}}

\begin{document}

\begin{abstract} The M\"obius energy of a curve is a topic of interest to physical knot theorists, harmonic analysts, and geometric analysts. In particular, there are many open questions regarding the gradient flow and critical points. The Gateaux derivative indicates the variation is dependent on curvature and torsion, and so to better understand the gradient, we investigate the family of helix curves, where the ratio of torsion to curvature is a constant proportional to the pitch.

We fix the radius of the helix, and study the coiling in both directions: as the helix unravels to a straight line, and as it coils infinitely tight. Specifically, we study the arclength-rescaled M\"obius energy density, which is a naturally tractable quantity under the M\"obius energy's chord-arc comparison of inverse-square laws. 

The asymptotics of the uncoiling helix, corresponding to an energy decay, can be proven with a short chain of estimates. However, the asymptotics of the helix as it coils infinitely tight, blowing up the energy, is a much more involved calculation. Our strategy for calculating the asymptotics, initially reminiscent of the work by Kim-Kusner, begins with a meromorphic extension of the integrand. However, proving the asymptotic equivalence requires a fundamentally distinct strategy because our integrand has infinitely many poles. 

\end{abstract}


\maketitle
\section{Introduction}
\subsection{An abbreviated summary of the M\"obius energy}
~\\
A fundamental question in knot theory asks ``What is the best way to tie a knot?" Its solution would be a corollary to two fundamental questions of geometry, which ask ``What is the best shape?" and ``How do you construct it?" To resolve these questions, there are quantities known as energy functionals which geometers use to compare shapes. Ideally, these energies act as topological barriers, indicating a qualitative distinction between shapes. 

A simple example of a topological barrier is the integral of curvature. The F\`ary-Milnor theorem states that if $\gamma \subseteq \R^3$ is a $C^2$ simple closed curve such that $\int_{\gamma} \kappa(s)ds < 4\pi$, then $\gamma$ is unknotted. However, this theorem has limited utility because it does not distinguish curves with nontrivial knot type, nor does it actually provide the algorithm for untangling a knot with low energy.

The account of Simon gives a thorough overview of several different knot energies, discussing their various strengths and weaknesses \cite{simon02}. An example of a more refined topological barrier is the M\"obius energy, first introduced as part of a family of functionals by O'Hara \cite{ohara91}, and later expounded upon by Freedman-He-Wang in their seminal paper \cite{freedman1994mobius}.  

Let $j,p \geq 1$, $I$ be an interval or a circle and let $\gamma: I \to \R^n$ parametrize a $C^2$ curve. For $s \in I$, define the pointwise energy, also referred to as the energy element or energy density, as
\begin{subequations}
\begin{align}
\label{ejp}
    E^{j,p}(\gamma,s) &\coloneq \int_I \left( \frac{1}{|\gamma(t) - \gamma(s)|^j} - \frac{1}{D(\gamma(t),\gamma(s))^j} \right)^p |\dot{\gamma}(t)|dt,
\end{align}
where $D$ denotes the intrinsic distance along $\gamma$. The $(j,p)$-O'Hara energy of $\gamma$ is given by
\begin{align}
    E^{j,p}(\gamma) \coloneq \int_I E^{j,p}(\gamma,s)|\dot{\gamma}(s)|ds.
\end{align}
\end{subequations}

Since the extrinsic Euclidean distance minimizes the intrinsic distance on $\gamma$, the integrand in \eqref{ejp} is nonnegative. The arclength elements ensure the energy is invariant under reparametrizations of $\gamma$. As the quantity $E^{j,p}(\gamma,s)|\dot{\gamma}(s)|$ which is integrated, we will see that this arclength-rescaled energy density is the more natural quantity to analyze.

We will exclusively work with $E^{2,1}$ for curves in $\R^3$, suppressing the superscripts. This configuration of parameters, known as the M\"obius energy, is special. First of all, the M\"obius energy is now invariant under Euclidean similarities. Perhaps the central property of the M\"obius energy is that it is finite for simple closed curves, but blows up for self-intersecting curves. Therefore, the gradient descent precludes changes in knot type via transverse strand crossing. A ``pull-tight" singularity in the limit of an energy minimizing sequence, where a nontrivial part of a composite knot ceases to exist, could still occur. Another central result of Freedman-He-Wang is that prime knot types have M\"obius-minimizing parametrizations. 

The M\"obius energy has a coherent history, but we do not intend to provide a full account here. Our understanding of the M\"obius energy over the past thirty years is characterized by a sequence of increasingly strong regularity theorems. In \cite{freedman1994mobius}, Freedman-He-Wang show that if $E(\gamma) < \infty$, then $\gamma$ is $C^{1,1}$. That is, $\gamma$ is $C^1$ with Lipschitz derivative. In \cite{he2000}, He expands on this result by proving critical points are $C^\infty$. In a subsequent paper, Blatt shows that $E(\gamma) < \infty$ is equivalent to the condition that $\gamma$ is a simple curve in the fractional Sobolev space $H^{\frac{3}{2}}(I,\R^n)$ \cite{blatt12}. Recently in \cite{blatt20}, Blatt applies the Cauchy method of majorants to yield the very strong result that if $\gamma$ is a critical point of the M\"obius energy, then $\gamma$ is real-analytic. This is the only analyticity result of a nonlocal differential operator to the author's knowledge. 

Our choice to study helix curves originates from a non-rigorous heuristic observation regarding the first variation. In \cite{freedman1994mobius}, the authors calculate the Gateaux derivative of the M\"obius energy at $\gamma$, which is a vector field along $\gamma$, as 
\begin{align}
    \label{mobius-gradient}
    \nabla E(\gamma) (t) = 2 \int_{\gamma} \left[ \frac{2 \text{proj}_{\dot{\gamma}(s)^{\perp}}(\gamma(s) - \gamma(t))}{|\gamma(s)-\gamma(t)|^2} - \kappa \textbf{N}_{\gamma(t)}\right] \frac{|\dot{\gamma}(s)|}{|\gamma(s)-\gamma(t)|^2}ds.
\end{align}
Here, $\text{proj}_{\dot{\gamma}(s)^\perp}$ denotes projection onto the plane orthogonal to the tangent $\dot{\gamma}$ and $\{\textbf{T},\textbf{N},\textbf{B}\}$ denotes the Frenet frame. At a fixed $t$, this integral instructs us to take the difference vector $\gamma(s) - \gamma(t)$, project away $\textbf{T}$, subtract away the normal curvature vector $\kappa\textbf{N}$, while rescaling lengths along the way according to an inverse power law. Then we average by integrating over $s$. Thus, the M\"obius gradient is liable to be large when difference vectors are close to the binormal $\textbf{B}$, particularly when either a few difference vectors near $\gamma(t)$ align with $\textbf{B}$ well, or a lot of far away difference vectors align with $\textbf{B}$ well. Helix curves fall into the latter case, as the difference vector of a long-distance pair points in nearly the same direction as the helical axis. For a general curve with a subset approximated by a helix, uncoiling the curve (which corresponds to simultaneously performing Reidemeister-I moves) will decelerate its M\"obius energy. By studying the asymptotics of the helix, we seek insight into how the M\"obius energy varies as a general curve with a sub-arc approximated by a helix coils or uncoils. 
\subsection{Overview of the main result and its proof}
~\\
A helix has infinite M\"obius energy, but it turns out that the energy density is finite. By invariance under rigid screw isometries, we have that the density is uniform across the whole helix. The pitch $\rho$ of a helix, which is the constant ratio of torsion to curvature, is its essential geometric parameter. Hence, we are led to study the relationship between pitch and energy density.

Specifically, a helix of pitch $2\pi\rho > 0$ (we will henceforth instead use the word ``pitch" to refer to $\rho$) is parametrized by $H_\rho(t) \coloneq (e^{it},\rho t)$ in $\R^3$. As $\rho \to \infty$, the helix unravels to a straight line, uniformly on compact subsets. As $\rho \to 0$, the helix coils infinitely tight, converging to a pathological limiting object which is not a curve. Notably, the torsion $\tau = \frac{\rho}{\rho^2 + 1} \to 0$ in both cases. Taking into account the Euclidean invariance of the M\"obius energy and rescaling by the arclength element $\sqrt{\rho^2 + 1}$, we are led to the real-valued function $I: \R^+ \to \R^+$ defined by
\begin{align*}
    I(\rho) &\coloneq \sqrt{\rho^2 + 1}E(H_\rho,0) \\
    &= \int_{-\infty}^{\infty} \left( \frac{\rho^2 + 1}{\rho^2 t^2 + 4\sin^2 \frac{t}{2}} - \frac{1}{t^2} \right) dt \\
    &\eqqcolon \int_{-\infty}^\infty M_\rho(t)dt.
\end{align*}
The rescaling isolates the regularizing $\frac{1}{t^2}$ term. Upon closer inspection, one can see that $M_\rho$ is indeed a continuous one-parameter family of real-analytic functions. In fact, as functions of a complex variable, the family is analytic at a neighborhood of the origin, which can be chosen independently of $\rho$. So while variations of the M\"obius energy require the Cauchy principal value, we need not concern ourselves with truncations. Furthermore, as the M\"obius energy is defined by a chord-arc comparison, we knew that the integrand would be nonnegative before setting it up. Deducing this from the formula of $M_\rho$ requires verification.

Of course, $\{M_\rho\}$ is a family of integrable functions. An elementary argument, relying on the calculation of a well-known definite integral, leads to our initial result.
\begin{reptheorem}{helix_pointwise_energy}
As $\rho \to \infty$,
\begin{align*}
    I(\rho) &\sim \frac{1}{\rho^2}.
\end{align*}
\end{reptheorem}
\noindent This theorem is proven by the sandwich inequality
\begin{center}
$\frac{C}{\rho^2 + 1} \leq I(\rho) \leq \frac{C}{\rho^2}$.    
\end{center}
However, this inequality tells us nothing about the $\rho \to 0$ asymptotics. The central result of this paper is
\begin{reptheorem}{irho_asymptotics}
    As $\rho \to 0$,
    \begin{align*}
        I(\rho) &\sim \frac{\log \frac{1}{\rho}}{\rho}.
    \end{align*}
\end{reptheorem}
A natural idea is to extend $M_\rho$ to the meromorphic function $M_\rho(z)$ and calculate $I(\rho)$ via a complex contour and the Cauchy integral formula. Indeed, this was precisely the approach of Kim-Kusner in their work on the M\"obius energy of torus knots $T_{p,q}$ \cite{kusner93} with varying radii $r$. The Mobius energy can be represented as a finite sum of residues in the unit disc, in a calculation credited to Kusner and Stengle. They prove existence of a critical $r$-value, and their numerical approximations suggest this critical value is unique. However, recent work by Blatt-Gilsbach-Reiter-von der Mosel claims proof of the existence of multiple critical points of every torus knot type \cite{blatt25}. 

The first difference we encounter is that there is no critical $\rho$-value: $I$ is strictly decreasing, as shown in Lemma \ref{Idecreasing}. This is why we turn our attention to the asymptotics, as opposed to the energy for any fixed and finite $\rho$. Next, to apply the Cauchy integral formula, we needed to apply a less-than-obvious family of contours, as depicted in Fig. \ref{Gamma_R}, which engulfs the upper-half of the complex plane. This justifies writing
\begin{align*}
    I(\rho) &= 2\pi i \sum\limits_{\Imz z > 0} \text{Res}(M_\rho,z).
\end{align*}
Unlike many contour integrals from a first course in complex analysis, there are infinitely many poles, and so an estimate with an arbitrarily large contour will not suffice. This calculation, while necessary, is narrow because the contour we use is not the only valid option. Therefore we have relegated this work to an appendix.

The problem of locating the poles brings us to our first technical challenge. We remark that the regularization removes the origin from the list of poles, ensuring $I(\rho)$ converges, but this does not simplify the problem of locating the poles. Nonzero poles of $M_\rho$ are in bijective correspondence with nonzero roots of the exponential polynomial
\begin{align*}
    E_\rho(z) = \rho^2z^2 + 4\sin^2 \frac{z}{2}.
\end{align*}
Specifically, we prove the following.
\begin{repprop}{erho_existence_uniqueness}
    For all $\rho > 0$ and $k \in \N$, $E_\rho$ has a unique root in the vertical half-strip $[(2k-1)\pi,2k\pi] \times \R^+$. This enumerates all the roots in the upper-right quadrant of $\C$.
\end{repprop}

By factoring, rescaling, and using symmetry, it suffices to describe the complex solutions of the equation $\sin z = i\rho z$ in the upper-half plane. One can see from a graph (see Fig. \ref{sinc_irho_plots}) that the roots are distributed with roughly linearly growing real parts and logarithmically growing imaginary parts. The method for proving an existence and uniqueness theorem is straightforward: apply an intermediate value theorem to the real and imaginary parts. The obstacle, however, is actually demonstrating the boundary conditions and differential inequalities are satisfied.

It turns out most of the $\rho$-independent numbers we need to work with are uncommon transcendentals. Specifically, we define the numbers $c_k$, which are the solutions of $\theta \tan \theta + 1 = 0$ in the interval $[(k-1)\pi,k\pi]$. We also make use of the critical point and maximum value of the function $u \mapsto u\sech u$, whose graph looks like the derivative of a Gaussian. This geometric and analytic approach traces back to the work of Polya and his student Schwengeler on the zeros of exponential polynomials \cite{polya1920,schwengeler1925}. 

The study of exponential polynomials is vast and highly developed, with many canonical theorems recorded in \cite{langer31} and more recently, \cite{hitw23}. However, many of the theorems regarding zeros of exponential polynomials describe only the asymptotic distribution, whereas we seek a statement describing all zeros. Furthermore, as $E_\rho$ is the sum of a polynomial and an exponential, results on pure exponential sums do not apply. The author did not find a mechanism within the exponential polynomial literature which directly yields a statement obviously equivalent to Prop. \ref{erho_existence_uniqueness}. As a result, we needed to formulate our own direct proof.

We remark that in the proof, we have a split regime whose boundary is roughly dependent on the product $k\pi\rho$, altering the chain of inequalities we use to show the boundary conditions and differential inequality in our intermediate value theorem (Lemma \ref{crossing}). We see situations such as one derivative being positive and the other derivative being negative over most of the interval $[(k-1)\pi,k\pi]$, except for a tiny sliver at the end. If one looks at the graphs, one derivative is obviously much larger than the other, indicating the desired differential inequality is clearly satisfied. However, some work was needed to rigorously prove this. The structural symmetries of trigonometric and hyperbolic trigonometric functions which permit the proof to go forward. We also remark that depending on the regime, the boundary condition and the differential inequality trade off in the effort needed to prove them, in such a way that the total required effort remains roughly constant.

By enumerating the roots $z_k = z_k(\rho)$ and taking into account the relevant symmetries, the calculus of residues yields
\begin{align*}
    I(\rho) &= -4\pi (\rho^2 + 1) \sum\limits_{k=1}^\infty \Imz \frac{1}{E'_\rho(z_k(\rho))}.
\end{align*}
As we expect, the roots of $E_\rho$ do not have a closed formula. This leads us to consider the approximate sequence and series
\begin{align*}
    w_k(\rho) &\coloneq 2\pi k + 2 i \arcsinh k\pi\rho \\
    \tilde{I}(\rho) &\coloneq -4\pi (\rho^2  +1) \sum\limits_{k=1}^\infty \Imz \frac{1}{E'_\rho(w_k(\rho))}.
\end{align*}
It was clear from the outset that $2\pi k$ was a strong approximation to the real part of $z_k$. So much so that we could use a $\rho$-independent formula and focus our attention on approximating the imaginary part. A direct calculation shows
\begin{align*}
    E_\rho(w_k(\rho)) &= 4\pi k \rho^2\left[ -\frac{\arcsinh^2 k\pi\rho}{k\pi} + 2i\arcsinh k\pi\rho\right],
\end{align*}
and for small $\rho$, we can see that we have approximated a solution of $E_\rho(z) = 0$ very well. From the discrete distribution of the poles given by Prop. \ref{erho_existence_uniqueness}, and the rigidity of holomorphic functions, we are assured that we are close to the actual $z_k(\rho)$.

As we now have a series written in closed form, we can calculate the asymptotics directly with a classical Euler-Maclaurin summation, which we do in Lemma \ref{Itilde-asymp}. Indeed, $\tilde{I}(\rho) \sim \frac{\log \frac{1}{\rho}}{\rho}$ as $\rho \to 0$. However, we must still rigorously prove the asymptotic equivalence of $I$ and $\tilde{I}$, which is the second technical challenge we encounter. The bridge between the true poles $z_k$ and the approximates $w_k$ is given by the following error estimate.
\begin{repprop}{zkrk_bound_thm}
    For all $\rho > 0$ sufficiently small and for all $k \in \N$, we have
    \begin{align*}
        |z_k(\rho) - w_k(\rho)| &\leq C \frac{\rho \arcsinh k\pi\rho}{\sqrt{(k\pi\rho)^2 + 1}}.
    \end{align*}
\end{repprop}
The proof involves the judicious choice of radii $r_k = \frac{2|E_\rho(w_k)|}{|E'_\rho(w_k)|}$, and applying the Rouch\'e theorem for $E_\rho$ in a disc around $w_k$. The required technique involves absorbing and controlling inequalities, which are ubiquitous in geometric analysis. Even though it is clear that the root we detect near $w_k$ is $z_k$, Prop. \ref{erho_existence_uniqueness} is logically essential to ruling all other possibilities. In this sense, the fusion of the harmonic analysis of the Polya-Schwengeler description of roots and the geometric-analytic technique used for the Rouch\'e hypotheses is mathematically essential to the execution of this helix calculation.

Now, we have that $I$ and $\tilde{I}$ are real-valued, but the strategy to work with complex variables has proven to be plausible so far, and so we expect to ultimately show an asymptotic equivalence of complex functions. However, even for nice complex functions $f(\rho) \sim g(\rho)$ does not necessarily imply $\Imz f(\rho) \sim \Imz g(\rho)$. In a preliminary section, we set up and prove Theorem \ref{transferthm}, a transfer theorem with sign and non-tangency hypotheses. The transfer theorem, due to its simplicity, may be of independent utility for problems requiring a bridge between the real and complex.

Its proof proceeds as expected. Furthermore, it is easy to show the $w_k$-series, due to its closed form, satisfies the sign and non-tangency conditions. However, due to symmetry, we should expect the $z_k$-series to also satisfy the hypotheses, which we show with Lemma \ref{jtransferthm}. While this is not necessary to conclude asymptotic equivalence, the proof requires us to derive a crucial perturbation of the form
\begin{align*}
    \frac{1}{E_\rho'(z_k)} &= \frac{1}{E'_\rho(w_k)}(1+\eta_k),
\end{align*}
where $\eta_k = \eta_k(\rho)$ is a sequence of complex numbers with bounds $|\eta_k| \leq C\rho$, uniformly in $k$. This controlling estimate, along with the sign and non-tangency conditions for the $w_k$-series, are precisely the ingredients which allow us to complete the proof of Theorem \ref{irho_asymptotics} via the following culminating asymptotic equivalence of infinite series:
\begin{repprop}{Jasymptotics}
    As $\rho \to 0,$
    \begin{align*}
        \sum\limits_{k=1}^\infty \frac{1}{E'_\rho(z_k(\rho))} &\sim \sum\limits_{k=1}^\infty \frac{1}{E'_\rho(w_k(\rho))}.
    \end{align*}
\end{repprop}

\subsection{Towards an asymptotics statement for more general curve flows}
~\\
While we have only worked directly with helices, our broader aim is to use them as a model for a general curve flow coiling infinitely tight near a point. That is, where the torsion-to-curvature ratio $\frac{\tau}{\kappa}$, which we call the local pitch $\rho$, tends to zero, as suggested by our heuristic for the first variation in Eq. \eqref{mobius-gradient}.

A potentially useful expression of the arclength-rescaled M\"obius energy density of $\gamma$ at $s_0 \in I$ is of the form
\begin{align*}
    E(\gamma,s_0)|\dot{\gamma}(s_0)| &= I(\rho(s_0)) + R(\gamma,s_0),
\end{align*}
where under suitable hypotheses, and a suitable definition for a curve flow to ``coil infinitely tight", the remainder $R(\gamma,s_0) = o(I(\rho))$. The central obstruction is the nonlocality of the M\"obius energy: the blowup of the energy is a consequence of the curve nearly returning to itself. A genuinely nonlocal hypothesis is required, but coiling is one way to quantify a near return whilst maintaining uniform local geometry. We leave the precise formulation and proof of this asymptotics theorem to future work, which would be significant in its own right because it would have to yield our theorem as a corollary. 

Little is known about the structure of M\"obius minimizing sequences of composite knots (knots which decompose as the connect sum of two nontrivial knots). The existence problem of minimizing representatives in a composite knot isotopy class remains open. The upshot would be that a pull-tight singularity in a knot flow which coils infinitely tight would blow up the M\"obius energy. This would impose necessary geometric constraints for a flow which decreases the M\"obius energy.  

Even though general $C^k$ curves preclude the use of our meromorphic methods, due to strong regularity results of the critical points shown by Blatt in \cite{blatt20}., curve flows expressed with analytic formulas are the main method to search for M\"obius critical knots, which is a topic of considerable interest due to open questions posed by Freedman-He-Wang and He \cite{freedman1994mobius,he2000}. In particular, if a curve flow is given in terms of a formula involving trigonometric polynomials, then the extrinsic distance term $|\gamma(s)-\gamma(t)|^2$ must also be a trigonometric polynomial, which means the techniques from this paper and those of Kim-Kusner remain available.

\section{Preliminaries on Asymptotics}
We briefly establish our conventions for asymptotics. The fundamental relationship we work with is asymptotic equivalence. We will typically refrain from other types of asymptotic notation or external asymptotic results.

\begin{definition}
    Let $x_n$ and $y_n$ be sequences of complex numbers. We say $x_n$ is \textbf{asymptotically equivalent to} $y_n$ \textbf{as} $n \to \infty$, denoted $x_n \sim y_n$, if and only if for all $\vep > 0$, there exists $N \in \N$ such that for all $n > N$,  $$|x_n - y_n| < \vep |y_n|.$$ 
\end{definition}
If we further assume that $y_n$ does not vanish infinitely often, then $x_n \sim y_n$ is equivalent to 
\begin{align*}
    \lim\limits_{n \to \infty} \frac{x_n}{y_n} &= 1.
\end{align*}
It is clear that asymptotic equivalence is indeed an equivalence relation on the set of of all sequences, and a sequence which vanishes infinitely often is never asymptotically equivalent to a sequence which does not.

For scalar-valued functions, we can also define asymptotic equivalence. In this paper, we will only be concerned with limits at zero or infinity. Of course, asymptotic equivalence of functions is a stronger property because the condition must hold for uncountably infinite $\rho$-values.

\begin{definition}
    Let $f,g: (0,\infty) \to \C$ be functions. We say $f$ is \textbf{asymptotically equivalent to} $g$ \textbf{as} $\rho \to 0$ (resp. $\rho \to \infty$), denoted $f(\rho) \sim g(\rho)$, if and only if for every $\vep > 0$, there exists $\rho_0 > 0$ such that for all $0 < \rho < \rho_0$, $|f(\rho)-g(\rho)| < \vep |g(\rho)|$ (resp. for every $\vep > 0$, there exists $R > 0$ such that for all $\rho > R$, $|f(\rho) - g(\rho)| < \vep |g(\rho)|$).
\end{definition}
If $g(\rho)$ is nonzero, proving $f(\rho) \sim g(\rho)$ is equivalent to proving $\frac{|f(\rho) - g(\rho)|}{|g(\rho)|} \to 0$ as $\rho$ approaches its respective limiting value.

We will need to bridge real and complex asymptotics, which leads us to set up and prove the following transfer theorem.

\begin{theorem}[Transfer Theorem] \label{transferthm}
    Let $F,G: (0,\infty) \to \C$ be complex functions which admit absolutely convergent series representations. That is, we have one parameter families of sequences of complex numbers $a_k(\rho)$ and $b_k(\rho)$ such that for all $\rho > 0$,
     \begin{align*}
    F(\rho) &\coloneq \sum_{k=1}^{\infty} a_k(\rho), \quad
    G(\rho) \coloneq \sum_{k=1}^{\infty} b_k(\rho),
\end{align*}
where all the series are absolutely convergent. Suppose there exists $c>0$ and $\rho_0 > 0$ such that for all $0 < \rho < \rho_0$ and $k \in \N$, the following conditions hold:
\begin{enumerate}
    \item (Sign condition) The imaginary parts of all $b_k(\rho)$ share the same sign and are nonzero.
    \item (Non-tangency) $\left| \Imz b_k(\rho) \right| \geq c|b_k(\rho)|.$
\end{enumerate}
Then if $F(\rho) \sim G(\rho)$ as $\rho \to 0$, we have that $\Imz F(\rho) \sim \Imz G(\rho)$ as $\rho \to 0$.
\end{theorem}

\begin{proof}
    Let $\vep > 0$ and let $\rho_0$ and $c$ be as above so both conditions are satisfied. Furthermore, assume $\rho_0$ is sufficiently small such that for all $0 < \rho < \rho_0$, $|F(\rho) - G(\rho)| < c\vep |G(\rho)|.$ Write $F_0(\rho) \coloneq \sum\limits_{k=1}^\infty \Imz a_k(\rho)$ and $G_0 \coloneq \sum\limits_{k=1}^\infty \Imz b_k(\rho)$, both of which are absolutely convergent real series.

    By the sign condition and non-tangency,
    \begin{align*}
        |G_0(\rho)| &= \sum\limits_{k=1}^\infty \left|\Imz b_k(\rho) \right| \\
        &\geq c\sum\limits_{k=1}^\infty |b_k(\rho)| \\
        &\geq c|G(\rho)|.
    \end{align*}
    Therefore, applying the trivial bound $\left|\Imz z\right| \leq |z|$, we have
    \begin{align*}
        |F_0(\rho) - G_0(\rho)| &\leq |F(\rho) - G(\rho)| \\
        &< c\vep |G(\rho)| \\
        &\leq \vep |G_0(\rho)|.
    \end{align*}
\end{proof}

\section{The M\"obius energy integral of a helix}
\subsection{Setup and initial observations}
~\\
For $\rho > 0$, a helix of pitch $2\pi\rho$ is parametrized by 
\begin{align*}
    H_\rho(t) &\coloneq (e^{it},\rho t) \\
    &= (\cos t, \sin t, \rho t).
\end{align*}
The M\"obius energy is invariant under Euclidean similarities, and $\rho$ exhausts all similarity classes of helices. Let us compute
\begin{align*}
    \left| H_\rho(t) - H_\rho(s) \right|^2 &= \left|e^{it} - e^{is}\right|^2 + \rho^2(t-s)^2 \\
    &= \left| e^{i(t-s)} - 1 \right|^2 + \rho^2(t-s)^2\\
    &= 4\sin^2\left( \frac{t-s}{2} \right) + \rho^2(t-s)^2.
\end{align*}
The arclength element is the constant $|\dot{H}_\rho| = \sqrt{1 + \rho^2}$. Therefore,
\begin{align*}
    D(H_\rho(t),H_\rho(s))^2 = (1 + \rho^2)(t-s)^2.
\end{align*}
Assembling all of this into \eqref{ejp}, we have
\begin{align*}
    E(H_\rho,s) &= \int_{-\infty}^\infty \left(\frac{\sqrt{1 + \rho^2}}{4\sin^2\left( \frac{t-s}{2} \right) + \rho^2(t-s)^2} - \frac{1}{\sqrt{1+\rho^2}(t-s)^2} \right)dt. 
\end{align*}
 By the invariance of $H_\rho$ by rigid screw motions, the energy density is constant in $s$. This implies $E(H_\rho) = \infty$, akin to how a line has infinite length. We are therefore confined to analyzing the energy density $E(H_\rho, 0)$.
 
 Rescaling the density allows us to isolate the regularization term in $t$. So, let us define the central quantity we study in this paper:
 \begin{align}
 \label{mrho}
     I(\rho) &\coloneq \sqrt{1 + \rho^2}E(H_\rho,0) \nonumber \\
     &= \int_{-\infty}^\infty \left( \frac{1 + \rho^2}{\rho^2 t^2 + 4\sin^2 \frac{t}{2}} - \frac{1}{t^2}\right)dt \nonumber \\
     &\eqqcolon \int_{-\infty}^{\infty} M_\rho(t) dt.
 \end{align}

\begin{remark}
    Truncated helices have finite energy. Let $C_N = \{ (x,y,z) \in \R^3: x^2 + y^2 \leq 1, |z| \leq N \}$ denote the closed cylinder with fixed height $2N > 0$. The helix segments $H_\rho \cap C_N$ will converge uniformly as $\rho \to \infty$ to a vertical straight line segment passing through $(0,0,1)$, which has zero energy. We can apply Theorem \ref{helix_pointwise_energy} to see
    \begin{align*}
        E(H_\rho \cap C_N) &= (1 + \rho^2)\int_{-\frac{N}{\rho}}^{\frac{N}{\rho}} \int_{-\frac{N}{\rho}}^{\frac{N}{\rho}} M_{\rho}(t-s) dsdt \\
        &\leq \sqrt{1 + \rho^2} \int_{\frac{-N}{\rho}}^{\frac{N}{\rho}} I(\rho) dt \\
        &\leq C \sqrt{1 + \rho^2}\int_{-\frac{N}{\rho}}^{\frac{N}{\rho}} \frac{dt}{\rho^2}\\ &= 2NC \frac{\sqrt{1 + \rho^2}}{\rho^3} \to 0, \text{ as $\rho \to \infty$.}
    \end{align*}
    Of course, this is exactly what lower-semicontinuity of the M\"obius energy predicts. By the absolute continuity of the M\"obius energy (see Eq. 1.6 in \cite{freedman1994mobius}), the truncation need not be symmetric, as was the case in this remark.
\end{remark}

\begin{figure}[h]
    \centering
    \includegraphics[width=4in]{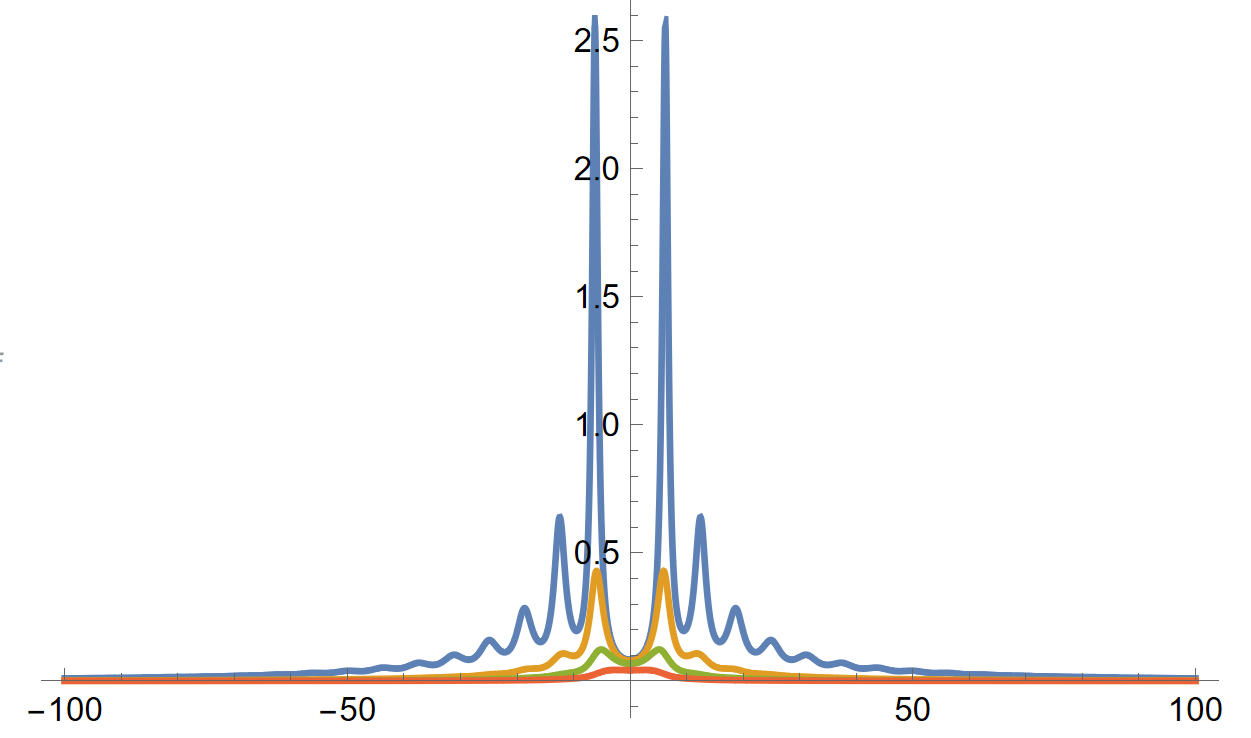}
    \caption{Plots of $M_\rho(t)$, for various $0 < \rho \leq 1$.}
\end{figure}

\begin{theorem}
\label{helix_pointwise_energy}
As $\rho \to \infty$,
\begin{align}
    I(\rho) \sim \frac{1}{\rho^2}.
\end{align}

\end{theorem}

\begin{proof}
    Let $\rho > 0$. Recall the definition $\sinc(t) \coloneq \frac{\sin t}{t}$, which is analytic and bounded above by $1$ on $\R$, and observe
    \begin{align}
        I(\rho) &=  \frac{1}{2} \int_{-\infty}^{\infty} \frac{t^2 - \sin^2 t}{t^2\left(\rho^2t^2 + \sin^2 t\right)}dt \nonumber\\
        &= \frac{1}{2} \int_{-\infty}^{\infty} \frac{1 - \sinc^2 t}{t^2\left(\rho^2 + \sinc^2 t\right)}dt. \label{sinc_integral}
    \end{align}
Hence
\begin{align*}
    I(\rho) &\leq \frac{1}{2\rho^2} \int_{-\infty}^\infty \frac{1 - \sinc^2 t}{t^2}dt \\
    &= \frac{\pi}{3\rho^2}.
\end{align*}
The prior definite integral is well known, appearing in \cite{medhurst65} and \cite{gr07}. Next, as  $\rho^2 + \sinc^2 t \leq \rho^2 + 1$, we have
\begin{align*}
    I(\rho) &\geq \frac{1}{2 (\rho^2 + 1)} \int_{-\infty}^{\infty} \frac{1 - \sinc^2 t}{t^2}dt \\
    &= \frac{\pi}{3(\rho^2 + 1)}.
\end{align*}
Again, we have used the same definite integral from \cite{gr07}.
\end{proof}

As we expect, $I(\rho)$ is monotonic. In fact, $I$ is differentiable and strictly decreasing.

\begin{lemma}\label{Idecreasing}
    For $\rho > 0$, we have
    \begin{align*}
        \frac{dI}{d\rho} &< 0.
    \end{align*}
\end{lemma}

\begin{proof}
    Differentiating under the integral sign yields
    \begin{align*}
        \frac{dI}{d\rho} &= \bigintsss_{-\infty}^{\infty} \frac{\rho(\sinc^2 t - 1)}{t^2\left(\rho^2 + \sinc^2 t\right)^2}dt.
    \end{align*}

\noindent This integrand is nonpositive on $\R$, and it is easy to verify the integral converges.
\end{proof}

Theorem \ref{helix_pointwise_energy} is equivalent to the bounds
\begin{align}
\label{irho_estimate}
    \frac{C}{\rho^2+1} \leq I(\rho) \leq \frac{C}{\rho^2},
\end{align}
\noindent The remainder of this section is concerned with proving the following theorem.

\begin{theorem}
\label{irho_asymptotics}
As $\rho \to 0$,
\begin{align*}
    I(\rho) \sim \frac{\log \frac{1}{\rho}}{\rho}.
\end{align*}
\end{theorem}

\subsection{Expansion of $I(\rho)$ into a series of residues} \label{residueseries}
~\\
In order to calculate the asymptotics of $I$, we rewrite $I$ as a discrete infinite series. We have that
\begin{align}
\label{irho_residueseries}
    I(\rho) &= 2\pi i \sum\limits_{\text{Im} z > 0} \text{Res }(M_\rho,z).
\end{align}

This equality must be justified. As there are countably infinite poles, proving \eqref{irho_residueseries} requires more than a general estimate on arbitrarily large contours.

We may write $M_\rho(t)$ with the complex variable $z$ a couple of different ways:
\begin{subequations}
\begin{align}
    M_\rho(z) &= \frac{\rho^2 + 1}{\rho^2z^2 +2 - 2\cos(z)} - \frac{1}{z^2} \label{mrho1} \\
    &=  \frac{\rho^2 + 1}{(\rho z)^2 + (2\sin(\frac{z}{2}))^2} - \frac{1}{z^2} \label{mrho2} \\
    &= \frac{\rho^2 + 1}{(2\sin({\frac{z}{2}}) + i\rho z)(2\sin(\frac{z}{2})-i\rho z)} - \frac{1}{z^2} \label{mrho3} \\
    &=  \frac{1 - \sinc^2(\frac{z}{2})}{z^2(\sinc(\frac{z}{2})+i\rho)(\sinc(\frac{z}{2})- i\rho)}. \label{mrho4}
\end{align}
\end{subequations}

Even though the expressions \eqref{mrho1}-\eqref{mrho4} indicate a possible singularity at $z = 0$, observe
\begin{align*}
     \frac{\rho^2 + 1}{(\rho z)^2 + (2\sin(\frac{z}{2}))^2} &= \frac{\rho^2+1}{\rho^2z^2 + (z^2 + O(z^4))} \\
     &= \frac{1}{z^2} \left(\frac{\rho^2+1}{(\rho^2 + 1) + O(z^2)}\right) \\
     &= \frac{1}{z^2}(1 + O(z^2)).
\end{align*}
So the second-order pole in this expression cancels with the $-\frac{1}{z^2}$ term in \eqref{mrho2}, which means the singularity at the origin is removable.

\begin{figure}
    \centering
    \includegraphics[width=0.5\linewidth]{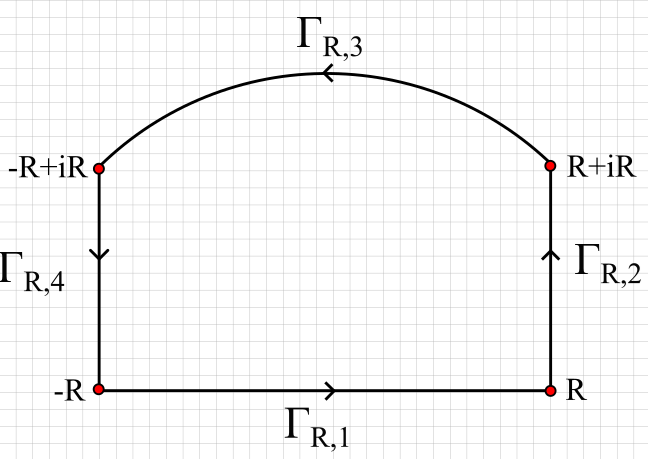}
    \caption{The oriented contour $\Gamma_R$ and its four subcontours. Here, $\Gamma_{R,3}$ is the segment of the circle of radius $R\sqrt{2}$ centered at $0$ joining the points $R+iR$ and $-R+iR$.}
    \label{Gamma_R}
\end{figure}

For $R > 0$, let $\Gamma_R$ be the counterclockwise oriented contour in the complex plane, as depicted in Fig. \ref{Gamma_R}, along with its four segments $\Gamma_{R,i}$, $i = 1,\dots, 4$. With a little bit of work shown in the appendix, we can show $\int_{\Gamma_{R,2}} + \int_{\Gamma_{R,4}}$ and $\int_{\Gamma_{R,3}}$ tend to zero as $R \to \infty$. 
Hence we may now assert
\begin{align*}
\int_{-\infty}^{\infty} M_{\rho}(t)dt &= \lim\limits_{R \to \infty} \int_{\Gamma_R} M_\rho(z)dz.
\end{align*}
And so, Eq. \eqref{irho_residueseries} follows from the Cauchy integral formula.
\subsection{The poles of $M_\rho$} \label{erho}
~\\
To calculate the sum of the residues for $M_\rho$, it helps to know where its poles are. As we expect, they have no closed form expression. Nonzero poles of $M_\rho$ are in bijective correspondence to nonzero roots of the exponential polynomial $E_\rho(z)$, defined as
\begin{align}
    E_\rho(z) \coloneq \rho^2z^2 + 4\sin^2 \frac{z}{2}.
\end{align}
We need to prove the following descriptive proposition of the roots.

\begin{prop} \label{erho_existence_uniqueness}
    For all $\rho > 0$ and $k \in \N$, $E_\rho$ has a unique root in the vertical half-strip $[(2k-1)\pi,2k\pi] \times \R^+$. This enumerates all the roots in the upper-right quadrant of $\C$.
\end{prop}

Though we are primarily concerned with small $\rho$, we opt to give the stronger statement which applies for all $\rho > 0$. The Hurwitz theorem gives a very rough description of the roots: there are two roots of $E_\rho$ counting multiplicity near the even integer multiples of $\pi$. This is not enough. There is no uniformity in $k$ and the Hurwitz theorem by itself does not tell us which half-plane the roots reside in.

To prove Prop. \ref{erho_existence_uniqueness}, we must set up some notation and preliminary lemmas. Observe
\begin{align}
    E_\rho(z) &= \left(2\sin \frac{z}{2} +i\rho z\right) \left( 2\sin \frac{z}{2} - i \rho z \right) \label{erho_factored} \\ 
    &\eqqcolon E_\rho^+(z)E_\rho^-(z). \nonumber
\end{align}
\noindent By Ritt's theorem, this factorization into irreducible exponential polynomials is unique. Our proof of Lemma \ref{erho_existence_uniqueness} will focus on $E_\rho^-$ with a rescaling of $z$. In other words, it is sufficient to describe the solutions of $\sin z = i\rho z$ in the upper-right quadrant of the complex plane. The argument involving $E_\rho^+$ is symmetric. Our argument involves repeated application of an intermediate value theorem. For shorthand, we call this technique the X-Principle, named for the crossing of the graphs. We omit its proof.

\begin{lemma}[X-Principle]
\label{crossing}
Let $[a,b]$ be a finite closed interval and suppose that $f,g:[a,b] \to \R$ are continuous functions that are differentiable in the open interval $(a,b)$ such that $f' > g'$. Suppose $f(a) < g(a)$ and $g(b) < f(b)$, where we permit the possibilities $f(a) = -\infty, g(a) = \infty, f(b) = \infty, g(b) = -\infty$, or any combination thereof. Then there exists a unique  $c \in (a,b)$ such that $f(c) = g(c)$.
\end{lemma}

Now we establish some notation. For $k \in \Z$, let $I_k$ denote the interval $[(k-\frac{1}{2})\pi,(k+\frac{1}{2})\pi]$ and $J_k$ denote the interval $[(k-1)\pi,k\pi]$. For $k > 0$, let $K_k \coloneq I_k \cap J_k$, and for $k < 0$, let $K_k \coloneq I_k \cap J_{k+1}$. For all $k \in \Z$, define the strips $S_k \coloneq I_k \times \R$ and $T_k \coloneq J_k \times \R$, and for $k \neq 0$, let $U_k \coloneq K_k \times \R$. Note that $K_0$ and $U_0$ are undefined. Thus, for $k > 0$, $U_k = S_k \cap T_k$, and for $k < 0$, $U_k = S_k \cap T_{k+1}$. The use of a $\pm$ superscript denotes the intersection of a strip with the upper and lower half-planes respectively.

Next, let $a^0: [-1,1] \to [-\frac{\pi}{2},\frac{\pi}{2}]$ be the principal branch of $\arcsin$. For $k \in \Z$, let $a^k:[-1,1] \to I_k$ denote the branch of $\arcsin$ whose range is $[(k-\frac{1}{2})\pi,(k+\frac{1}{2})\pi]$. Due to the oddity of $\sin$, the branches alternate. By this, we mean $a^{k}(u) = (-1)^ka^0(u)+k\pi$ for all $u \in [-1,1], k \in \Z$. Each $a^k$ is a left inverse of the restriction of $\sin$ to $I_k$. 

We also need to define some constants along with a function which comes up repeatedly. For $u \in \R$, let $\Sigma(u) \coloneq u\sech u$. The derivative of $\Sigma$ is
\begin{align}
\Sigma'(u) &= (1 - u\tanh u)\sech u . \label{dSigma}
\end{align}
\noindent Let $\beta$ be the positive critical point of $\Sigma$, with the corresponding critical value $\alpha$, which is the global maximum value of $\Sigma$. Then $\beta \approx 1.1997$ and $\alpha = \Sigma(\beta) \approx 0.6627$. From \eqref{dSigma} have the following two identities:
\begin{subequations}
\label{alphabeta}
\begin{align}
\beta \sech \beta &= \alpha \label{bsechba} \\ 
\beta \tanh \beta &= 1. \label{btanhb1}
\end{align}
\end{subequations}

The range of $\Sigma$ is $[-\alpha,\alpha]$, and when restricted to the real subsets $\{|u| \leq \beta\}$ or $\{|u| \geq \beta\}$, $\Sigma$ is injective. Let $\sigma:[-\alpha,\alpha] \to (-\infty,-\beta] \cup [\beta, \infty)$ and $\varsigma:[-\alpha,\alpha] \to [-\beta,\beta]$ be the corresponding left inverses of $\Sigma$. We note that $\sigma$ is discontinuous at $0$, with $\lim\limits_{u \to 0^{\mp}} \sigma(u) = \pm \infty$. On the subdomains $(-\alpha,0)$ and $(0,\alpha)$, $\sigma$ is strictly decreasing and smooth. Meanwhile, $\varsigma$ is a continuous strictly increasing function which is smooth between the endpoints. It turns out the only arguments we will place inside $\sigma$ are nonnegative, so without further comment, we will restrict $\sigma$ to $[0,\alpha]$.

The following lemma, based on some elementary observations of the graphs of $\tan \theta$ and $\frac{-1}{\theta}$ in $\R^2$, concerns a sequence of transcendental constants.
\begin{lemma}
\label{xtanx}
For each $k \in \Z, k \neq 0$, there is a unique solution $c_k \in K_k$ to the equation $\theta\tan \theta + 1 =0$. Furthermore, $c_k \sim k\pi$ as $k \to \infty$.
\end{lemma}

\begin{proof}
 Recall that the graph of $\tan \theta$ has vertical asymptotes at the boundaries of each $I_k$. Between each asymptote, $\tan \theta$ is strictly increasing from $-\infty$ to $\infty$, periodically. 

    Without loss of generality, assume $k > 0$. The function $-\frac{1}{\theta}$ is also strictly increasing for $\theta > 0$, tending to $0$ from below as $\theta \to \infty$. For all $\theta > 1$, $\frac{d}{d\theta}(\frac{-1}{\theta}) < 1$, whilst $\frac{d}{d\theta} (\tan \theta) = \sec^2 \theta \geq 1$ for all $\theta$. By the X-Principle, the graphs of $\tan \theta$ and $\frac{-1}{\theta}$ will intersect at a unique point between each vertical asymptote of $\tan \theta$. This is enough to claim the existence and uniqueness of $c_k$, but so far we can only conclude $c_k \in I_k$.

    A simple numerical search shows $c_1 \approx 2.7984$, which is precise enough to prove $c_1 \in K_1 = [\frac{\pi}{2},\pi]$. If we show that the sequence $k\pi - c_k$ is decreasing, we will prove $c_k \in K_k$ for all $k \geq 1$. Equivalently, we may show $c_{k+1} - c_k > \pi$. Suppose $c_{k+1} - c_k \leq \pi$. Since $\tan$ is $\pi$-periodic and strictly increasing within each period, with $c_k$ and $c_{k+1}$ lying in adjacent period domains, we can deduce $\tan c_{k+1} \leq \tan c_k$. However, as $c_k$ and $c_{k+1}$ both satisfy $\theta \tan \theta + 1 = 0$, we can deduce $\frac{-1}{c_{k+1}} \leq \frac{-1}{c_k}$. This contradicts the fact that $c_k < c_{k+1}$.
    
    We now prove asymptotic equivalence. Since the graph of $\frac{-1}{\theta}$ tends to the $x$-axis from below as $\theta \to \infty$, we conclude the ordinate of the intersection point within $U_k$ will tend to $0$ from below as $k \to \infty$. As $\tan \theta = 0$ precisely when $\theta$ is an integer multiple of $\pi$, and $k\pi$ is the only such value in $K_k$, we have that $k\pi - c_k$ is a monotonically decreasing sequence of positive numbers tending to zero. As $c_k$ is uniformly bounded away from zero, $\lim\limits_{k \to \infty} k\pi - c_k =0$ implies $\lim\limits_{k \to \infty} \frac{k\pi}{c_k} = 1$.
\end{proof}

By the limit comparison test with the harmonic series, we obtain the following corollary for free. However, we do not use it later. 
\begin{corollary} Let $a_k = k\pi - c_k$. Then $\sum\limits_{k=1}^\infty |a_k|^p$ converges for all $p > 1$.
\end{corollary}

It is worth noting that the sequence $a_k$ is monotonically decreasing, and so $\lim\limits_{p \to \infty} ||a_k||_{\ell^p} = ||a_k||_{\ell^\infty} = a_1 \approx 0.3432.$ 

Finally, we present a simple bound which will be crucial later. Recall that $\arcsinh u = \int_0^u \frac{dt}{\sqrt{1+t^2}} =  \log( u + \sqrt{u^2 + 1})$, which is increasing, smooth on $\R$, with $\arcsinh 0 = 0$.

\begin{lemma}\label{arcsinhu_overu}
For all $u \in \R / \{0\}$, we have $0 < \frac{\arcsinh u}{u} < 1$.
\end{lemma}
\begin{proof}
    Let $u > 0$. Notice that $0 < \frac{1}{\sqrt{1+t^2}} < 1$ for all $t \in (0,u]$. Integrate in $t$ over $[0,u]$ to get the desired result. The case for $u < 0$ is similar.
\end{proof}

\subsection{Proof of Proposition \ref{erho_existence_uniqueness} \'a la Polya-Schwengeler}
~\\
Let $\rho > 0$. We will locate the roots of the entire holomorphic function $f(z) = \frac{\sin z}{z} - i\rho$. Note the rescaling of $z$ from the definition of $E_\rho^-$ given Eq. \eqref{erho_factored}. After this rescaling, our task is to prove existence and uniqueness of a root in each of the strips $U_k^+$, where $k$ is a positive even integer.

Obviously, the set of roots is isolated and countable. Write $z \coloneq x + iy$, and consider the equation $f(z) = 0$.  As $f(0) \neq 0$, we may assume $z \neq 0$. By equating real and imaginary components, we get the system of equations 
\begin{subequations}
\begin{align}
    \rho x &= \cos x\sinh y. \label{GammaS} \\
    -\rho y &= \sin x\cosh y. \label{GammaC} 
\end{align}
\end{subequations}

Label the solution curves of Eqs. \eqref{GammaS} and \eqref{GammaC} in $\R^2 \cong \C$ by $\Gamma_S$ and $\Gamma_C$ respectively (the letters $S$ and $C$ stand for $\sinh$ and $\cosh$). All notation involving a $\Gamma$ will depend on $\rho$, which we suppress. 

We can write \eqref{GammaS} as 
\begin{align}
\label{GammaS_function}
    y &= \Gamma_S(x) = \arcsinh(\rho x \sec x)
\end{align}
The only issue with this rearrangement is the possible division by zero when dividing by $\cos x$, but if $\cos x = 0$, then \eqref{GammaS} implies $x = 0$, and in turn \eqref{GammaC} implies $y = 0$, and so we have a contradiction. In other words, there are no roots of $f$ on the boundaries of the strips $S_k$.

Let $\Gamma_{S,k} \coloneq \Gamma_S \cap S_k$. As mentioned earlier, $\arcsinh$ is well-defined, smooth, and strictly increasing on $\R$ with $\arcsinh 0 = 0$. Therefore, the vertical asymptotes of $\Gamma_S$ are the same as those of the graph of $x \mapsto \rho x \sec x$. These asymptotes occur at $x = (k + \frac{1}{2})\pi, k \in \Z$, which are the boundaries of the $S_k$.  Obviously the locations of these asymptotes are independent of $\rho$. In between asymptotes, $\Gamma_S$ is the graph of a smooth function. Consequently, each $\Gamma_{S,k}$ is a smooth curve with one connected component. With the exception of $I_0$, $\rho x \sec x$ is either strictly positive or strictly negative for all $x$ in a given $I_k$.  As $\sgn \arcsinh u = \sgn u$ for all $u \in \R$, the branches of $\Gamma_S$ (and hence the zeros in question) alternate between the upper and lower half-planes every other strip, in the same manner as $\rho x \sec x$. In other words,
\[
\begin{array}{ll}
    \Gamma_{S,k} \subseteq S_k^+, &  k \in \{2,4,6,\dots\} \cup \{-1,-3,-5,\dots\}  \\
    \Gamma_{S,k} \subseteq S_k^-, &  k \in \{1,3,5,\dots\} \cup \{-2,-4,-6,\dots\}. \\
\end{array} 
\]

In $S_0$, $\Gamma_S(x)$ is increasing, taking both positive and negative values. To prove this theorem, we apply the X-Principle to show $\Gamma_C$ intersects $\Gamma_S$ in a unique point within each $U^{\pm}_k$, with the sign chosen appropriately. 

We can differentiate \eqref{GammaS_function} to yield
\begin{align}
\label{GammaS_xderivative}
    \frac{d}{dx} \Gamma_{S,k}(x) &= \frac{\rho(x \tan x + 1)\sec x}{\sqrt{\rho^2x^2\sec^2 x + 1}}.
\end{align}
All of the factors in \eqref{GammaS_xderivative} are nonvanishing on $I_k$, save for $x\tan x + 1$. By Lemma \ref{xtanx}, there is a unique solution $c_k \in K_k \subseteq I_k$ satisfying $c_k \tan c_k + 1 = 0$. 

For the rest of this proof, assume $k$ is positive and even. For all $x \in I_k$, we have that $\frac{d}{dx} \Gamma_{S,k}(x) < 0$ when $x < c_k$, and $\frac{d}{dx} \Gamma_{S,k}(x) \geq 0$ when $x \geq c_k$. 

For $x \in [c_k,k\pi]$, we seek an upper bound of $\frac{d}{dx}\Gamma_{S,k}(x)$ in order to apply the X-Principle. For $x \in [c_k, k\pi]$, we have that $1 = \sec k\pi \leq \sec x \leq \sec c_k$. Furthermore, $x \tan x + 1 \leq 1$. Assembling these bounds into \eqref{GammaS_xderivative} yields
\begin{align}
\sup\limits_{x \in [c_k,k\pi]} \frac{d}{dx}\Gamma_{S,k}(x) &\leq \frac{\rho \sec c_k}{\sqrt{\rho^2 c_k^2 + 1}} \nonumber \\
&< \frac{\sec c_k}{c_k}. \label{supdGammaS}
\end{align}

Now we will examine $\Gamma_C$. Na\"ively, we could write \eqref{GammaC} as
\begin{subequations}
\begin{align}
\label{GammaC_function}
    x &= \Gamma_C(y) = -\arcsin(\rho y\sech y).
\end{align}
Unlike $\sec$, $\sech$ is well-defined on all of $\R$. And unlike $\arcsinh$, $\arcsin$ is multivalued with domain $[-1,1]$. Nevertheless, we may still separate variables and write either
\begin{align}
\label{GammaC_separated}
    -\rho y \sech y &= \sin x
\end{align}
\noindent or
\begin{align}
\label{GammaC_separated2}
    y\sech y &= -\frac{\sin x}{\rho},
\end{align}
\end{subequations}
without removing or introducing erroneous solutions because $\cosh$ is nonvanishing on $\R$.

Due to the $\rho$-bifurcation of $\Gamma_C$ depicted in Fig. \ref{sinc_irho_plots}, we divide into two cases which will dictate whether we use Eq. \eqref{GammaC_separated} or \eqref{GammaC_separated2}.

\noindent\textbf{Case 1, small $\rho$:}

First, suppose $\rho < \frac{1}{\alpha} \approx 1.5089$. For this case, we will write $\Gamma_{C,k} \coloneq \Gamma_C \cap S_k$. Then $\rho \Sigma(y)\in (-\rho\alpha,\rho\alpha) \subseteq[-1,1]$ for all $y \in \R$, which means that assuming $x \in I_k$, we may apply $a^k$ to both sides of \eqref{GammaC_separated}, and conclude $\Gamma_{C,k} = \{-a^k(\rho y \sech y) + iy : y \in \R\} \subseteq S_k$. To justify this equality of sets, we use the fact that each $a^k$ is one-to-one. So similarly to $\Gamma_{S,k}$, we can indeed conclude that $\Gamma_{C,k}$ is the graph of a smooth function, albeit in the $y$ variable as it ranges over $\R$.  As $\Sigma(y) \to 0$ as $|y| \to \infty$, we see $\Gamma_{C,k}$ has a vertical asymptote at $\{x = k\pi\}$. Likewise, by our assumption $\rho < \frac{1}{\alpha}$ and the fact $a^k$ only differs from $a^0$ by an additive constant, we are justified in differentiating \eqref{GammaC_function} to yield
\begin{align}
    \frac{d}{dy} \Gamma_{C,k}(y) &= -\frac{d}{dy} a^k(\rho\Sigma(y)) \nonumber\\
    &= -\frac{\rho\Sigma'(y)}{\sqrt{1 - \rho^2\Sigma(y)^2}} \nonumber\\
    &= \frac{\rho (y \tanh y - 1) \sech y}{\sqrt{1 - \rho^2 y^2 \sech\mathstrut^2 y}}. \label{GammaC_yderivative}
\end{align}

All of the factors of \eqref{GammaC_yderivative} are well-defined and nonvanishing for all $y \in \R$, except for $y\tanh y - 1$. We should also note the independence from $k$. Again, we have just applied the assumption that $\rho < \frac{1}{\alpha}$. So the sign of the derivative at a given $y \in \R$ is determined by the sign of $y \tanh y - 1$. The $y$-derivative is negative when $0 \leq y < \beta$, and positive when $y > \beta$.

Shifting gears momentarily to introduce notation, for any $0 < \rho < \frac{1}{\alpha}$, let $h_\rho \coloneq \frac{1}{\pi}a^0(\rho\alpha)$. Notice that $0 < h_{\rho} < \frac{1}{2}$ for all valid $\rho$. Thus, $\Gamma_{C,k}^+ \subseteq [(k-h_\rho)\pi,k\pi] \times \R^+ \subseteq U_k^+$. 

We will apply the X-Principle to $\Gamma_{S,k}$ and $\Gamma_{C,k}$ expressed as the graphs of functions of $x \in [(k-h_\rho)\pi,k\pi]$. As $\Gamma_{C,k}(y)$ has a single critical point at $y = \beta$, the inverse function theorem tells us $\Gamma_{C,k}^+$ is the union of the graphs of two functions of $x \in [(k-h_\rho)\pi,k\pi)$ whose ranges are $(0,\beta]$ and $[\beta,\infty)$ respectively. Now we further divide into subcases based on $\rho$, determining which of the two functions we will use for the X-Principle.

\noindent \textbf{Case 1a, small $\rho$, but not too small:} 

Suppose $(k-h_\rho)\pi < c_k$.  We will use the branch $\Gamma_{C,k} \cap \{y \geq \beta\}$. Since \eqref{GammaC_yderivative} tells us $\frac{d}{dy}\Gamma_{C,k}(y) > 0$ for $y > \beta$, the implicit function theorem allows $\Gamma_{C,k} \cap \{y \geq \beta\}$ to be expressed as the graph of a differentiable function of $x$ on the open interval $((k-h_\rho)\pi,k\pi)$. With an abuse of notation, we denote this function by $\Gamma_{C,k}(x)$. From \eqref{bsechba}, we can see this function extends continuously such that $\Gamma_{C,k}((k-h_\rho)\pi) = \beta$. Altogether, we can see that $\Gamma_{C,k}(x)$ is a continuous strictly increasing bijection from $[(k-h_\rho)\pi,k\pi)$ onto $[\beta,\infty)$.

To satisfy the boundary condition of the X-Principle, we claim $\Gamma_{S,k}$ lies above the line $\{y = \beta\}$. From \eqref{GammaS_xderivative}, we can see the minimal ordinate of $\Gamma_{S,k}$ is exactly $\arcsinh(\rho c_k \sec c_k)$. Dividing \eqref{btanhb1} by \eqref{bsechba} shows $\sinh \beta = \frac{1}{\alpha}$. Since $\sinh$ is increasing, we may apply it to both sides of our desired inequality, $\beta < \arcsinh(\rho c_k \sec c_k)$, and rearrange to see that our desired boundary condition is equivalent to $\cos c_k < \rho\alpha c_k$, which we will now show.

As assumed, $k\pi - c_k < \arcsin(\rho\alpha)$. As $\sin$ is increasing on $[0,\frac{\pi}{2}]$, we have that $\sin(k\pi - c_k) < \rho\alpha$. Hence,
\begin{align*}
\rho\alpha &> \sin(k\pi - c_k) \\
&= \cos c_k\sin k\pi - \cos k\pi\sin c_k \\
&= -\sin c_k.
\end{align*}
\noindent From the equation $c_k\tan c_k + 1 = 0$, we have that $-\sin c_k = \frac{\cos c_k}{c_k}$. Therefore, $\cos c_k < \rho\alpha c_k$ and we conclude the boundary condition of the X-Principle is satisfied.

Next, we will satisfy the differential inequality for the X-Principle. As we saw earlier, $\frac{d}{dy} \Gamma_{C,k}(y) > 0$ for $y > \beta$. Again, by the implicit function theorem, $\frac{d}{dx} \Gamma_{C,k}(x) > 0$ on the open interval $((k-h_\rho)\pi,k\pi)$. From \eqref{GammaS_xderivative}, we can see that $\frac{d}{dx}\Gamma_{S,k}(x) \leq 0$ for $x \in [(k-h_\rho)\pi,c_k]$, and so the only possible issue is that $\Gamma_{C,k}(x)$ and $\Gamma_{S,k}(x)$ both have nonnegative $x$-derivatives on $[c_k,k\pi]$.

Now, recall $\Sigma$ is a continuous, strictly decreasing bijection from $[\beta, \infty)$ onto $(0,\alpha]$. Let $y_\rho$ denote the number such that $c_k +iy_\rho \in \Gamma_{C,k} \cap \{y \geq \beta\}$. Then using our notation from earlier, $y_\rho = \sigma(\frac{-\sin c_k}{\rho})$. Note that we have just seen that $0 < -\frac{\sin c_k}{\rho } < \alpha$, and so $y_\rho$ is well-defined.  By the inverse function theorem and \eqref{GammaC_yderivative}, 
\begin{align*}
    \inf\limits_{x \in [c_k,k\pi)} \frac{d}{dx} \Gamma_{C,k}(x) &=  \inf\limits_{y \geq y_\rho} \left( \frac{d}{dy} \Gamma_{C,k}(y) \right)^{-1} \nonumber \\
    &= \inf\limits_{y \geq y_\rho} \frac{\sqrt{1 - \rho^2 y^2 \sech\mathstrut^2 y}}{\rho ( y\tanh y-1) \sech y}.
\end{align*}
\noindent For $y \geq y_\rho$, observe $\sqrt{1 - \rho^2 y^2 \sech\mathstrut^2 y} \geq \sqrt{1-\sin\mathstrut^2 c_k} = \cos c_k$. Furthermore, for $y \geq y_\rho,$
\begin{align*}
    \frac{\cosh y}{\rho(y\tanh y - 1)} &\geq \frac{\cosh y}{\rho y\tanh y} \\
    &= \frac{\coth y}{\rho \Sigma(y)} \\
    &\geq \frac{1}{\rho\Sigma(y_\rho)} \\
    &= \frac{-1}{\sin c_k} \\
    &= c_k \sec c_k.
\end{align*}
\noindent Note that we have used the fact that $\Sigma$ decreases on $[\beta, \infty)$. And now we can conclude 
\begin{align}
    \inf\limits_{x \in [c_k,k\pi)} \frac{d}{dx} \Gamma_{C,k}(x) &\geq c_k. \label{infdGammaC}
\end{align}
From \eqref{supdGammaS} and \eqref{infdGammaC}, it now suffices to prove $c_k \geq \frac{\sec c_k}{c_k}$ in order to show the X-Principle is satisfied on $[c_k, k\pi)$. Indeed, recall from Lemma \ref{xtanx} that $|c_k - k\pi|$ is a decreasing sequence of numbers in $[0,\pi]$. Hence as $k$ is a positive even integer, $c_k^2\cos c_k = c_k^2\cos(c_k - k\pi) \geq c_2^2\cos c_2.$ A simple numerical check verifies $c_2^2 \cos c_2 \approx 36.9794 > 1$, now by the X-Principle, we deduce the existence of a unique intersection point in $U_k^+$.

\noindent \textbf{Case 1b, extremely small $\rho$:}

Now suppose $c_k \leq (k-h_\rho)\pi$.  This implies $\Gamma_{C,k}^+ \subseteq [c_k,k\pi]\times \R^+$. In this case, we will use the branch $\Gamma_{C,k} \cap \{0 \leq y \leq \beta\}$. From \eqref{GammaC_yderivative}, we have that $\frac{d}{dy}\Gamma_{C,k}(y) < 0$ for $0 \leq y < \beta$. Again with an abuse of notation, we denote the corresponding function of $x \in [(k-h_\rho)\pi,k\pi]$ by $\Gamma_{C,k}(x)$. By the inverse function theorem, $\frac{d}{dx}\Gamma_{C,k}(x) < 0$.

The differential inequality is trivially satisfied since $\frac{d}{dx}\Gamma_{C,k}(x) < 0 < \frac{d}{dx}\Gamma_{S,k}(x)$ for $x \in ((k-h_\rho)\pi,k\pi)$. Likewise, $\Gamma_{C,k}(k\pi) = 0$ whilst $\Gamma_{S,k}(k\pi)$ is some positive number. All that remains is to satisfy the other boundary condition $\Gamma_{S,k}((k-h_\rho)\pi) < \Gamma_{C,k}((k-h_\rho)\pi) = \beta$. 

Our desired boundary condition is equivalent to the inequality $\rho (k-h_\rho)\pi \sec((k-h_\rho)\pi) < \sinh \beta = \frac{1}{\alpha}$. The function $x \mapsto x\sec x$ has critical points at the $c_k$'s and is increasing on $[c_k,(k + \frac{1}{2})\pi)$, since we assumed $k$ is positive and even. Now see that this subcase's assumption is equivalent to $\rho < \frac{\cos c_k}{c_k\alpha}$. Therefore, we have satisfied the boundary condition, and by the X-Principle, we conclude the existence of a unique intersection point in $[(k-h_\rho)\pi,k\pi] \times \R^+ \subseteq U_k^+$. 

To see that the branch $\Gamma_{C,k} \cap \{y \geq \beta\}$ does not intersect $\Gamma_{S,k}$, we note that we get the same derivative bounds from the prior subcase, except $\Gamma_{C,k} \cap \{y \geq \beta\}$ starts above $\Gamma_{S,k}$, and due to its larger $x$-derivative, will never intersect with $\Gamma_{S,k}$.

\noindent \textbf{Case 2, large $\rho$:}

Suppose $\rho \geq \frac{1}{\alpha}$. Our description of $\Gamma_S$ as a function of all $x \in \R$, save for the vertical asymptotes, is unaffected. Looking at \eqref{GammaC_separated2}, we note $|\frac{-\sin x}{\rho}| \leq \alpha$ for all $x \in \R$. Therefore, we may apply $\sigma$ and $\varsigma$ to both sides of \eqref{GammaC_separated2} and conclude $\Gamma_C = \{x + i \sigma(\frac{-\sin x}{\rho}): x \in \R\} \cup \{x + i\varsigma(\frac{-\sin x}{\rho}): x \in \R\} \eqqcolon \Gamma_{C,\sigma} \cup \Gamma_{C,\varsigma}$.

First, we show $\Gamma_{S,k} \cap \Gamma_{C,\varsigma} = \emptyset$. Since the range of $\varsigma$ is $[-\beta,\beta]$, it would suffice to show $\Gamma_{S,k}$ lies above the line $\{y = \beta\}$. We note that $c_k\sec c_k$ is an increasing sequence of positive numbers as $k$ ranges over the positive even integers. To see this, observe $c_k\sec c_k = -\csc c_k = -\csc(c_k - k\pi)$. By Lemma \ref{xtanx}, $c_k - k\pi \nearrow 0$. Then we can recall $-\csc$ is increasing on $[-\frac{\pi}{2},0)$. Now recalling that $\arcsinh$ is increasing, we can see that the minimal ordinate of $\Gamma_{S,k}$ is $\arcsinh(\rho c_k \sec c_k) \geq \arcsinh(\rho c_2 \sec c_2) \geq \arcsinh(\frac{1}{\alpha}c_2\sec c_2 )$. A numerical check confirms $\arcsinh (\frac{1}{\alpha} c_2\sec c_2) \approx 2.9323 > 1.1997 \approx \beta$.

For this case, we will relabel $\Gamma_{C,k} \coloneq \Gamma_{C,\sigma} \cap T_k$. That is, $\Gamma_{C,k}$ is the graph of the function $\Gamma_{C,k}(x) \coloneq \sigma(-\frac{\sin x}{\rho})$ over $x \in [(k-1)\pi,k\pi]$. Since $-\frac{\sin x}{\rho} > 0$ for $(k-1)\pi < x < k\pi$ and $\lim\limits_{x \to 0^+} \sigma(x) = +\infty$, we can see that the lines $\{x = (k-1)\pi\}$ and $\{x = k\pi\}$ are vertical asymptotes of $\Gamma_{C,k}$. To finish the proof, we apply the X-Principle for $\Gamma_{S,k}(x)$ and $\Gamma_{C,k}(x)$ over the interval $K_k = [(k-\frac{1}{2})\pi,k\pi]$.

Since $\Gamma_{S,k}$ has a vertical asymptote at $\{x = (k-\frac{1}{2})\pi\}$ and $\Gamma_{C,k}$ has a vertical asymptote at $\{x = k\pi\}$, the boundary conditions are satisfied. Observe that 
\begin{align}
    \frac{d}{dx}\Gamma_{C,k}(x) &= -\sigma'\left( \frac{-\sin x}{\rho} \right) \frac{\cos x}{\rho} \label{GammaC_xderivative}.
\end{align}
\noindent We can immediately see from \eqref{GammaC_xderivative} that $\frac{d}{dx}\Gamma_{C,k}(x) > 0$ for $x \in ((k-\frac{1}{2})\pi,k\pi)$. Our prior calculation from \eqref{GammaS_xderivative} shows $\frac{d}{dx}\Gamma_{S,k}(x) < 0$ on $((k-\frac{1}{2})\pi,c_k)$, and so from \eqref{supdGammaS}, it remains to prove 
\begin{align*}
    \inf\limits_{x \in [c_k, k\pi)}  \frac{d}{dx}\Gamma_{C,k}(x)  > \frac{\sec c_k}{c_k}.
\end{align*}

Let $u_\rho \coloneq \sigma(\frac{-\sin c_k}{\rho})$. By the inverse function theorem, for any $0 < v < \alpha$ and $u > \beta$ such that $v = \Sigma(u)$,
\begin{align*}
    \sigma'(v) &= \frac{1}{\Sigma'(u)} \\
    &= \frac{\cosh u}{1 - u\tanh u}.
\end{align*}
\noindent Therefore,
\begin{align*}
    \inf\limits_{x \in [c_k,k\pi)} -\sigma'\left( \frac{-\sin x}{\rho} \right) &= \inf\limits_{u \geq u_\rho} \frac{\cosh u}{u\tanh u - 1} \\
    &\geq \inf\limits_{u \geq u_\rho} \frac{\cosh u}{u\tanh u} \\
    &\geq \frac{\coth u_\rho}{\Sigma(u_\rho)} \\
    &= \rho c_k \sec c_k
\end{align*}
As $\frac{\cos x}{\rho} \geq \frac{\cos c_k}{\rho}$ for $x \in [c_k,k\pi)$, from \eqref{GammaC_xderivative} we can conclude 
\begin{align}
     \inf\limits_{x \in [c_k, k\pi)}  \frac{d}{dx}\Gamma_{C,k}(x) \geq c_k. \label{infdGammaC-largerho}
\end{align}
As $c_k > c_k\cos c_k > \frac{\sec c_k}{c_k}$, \eqref{supdGammaS} and \eqref{infdGammaC-largerho} show the differential inequality is satisfied, and so by the X-Principle, there is a unique intersection point in $U_k^+$. $\square$

\begin{figure}[h]
\captionsetup[subfigure]{justification=centering}
    \centering
    \begin{subfigure}{0.3\textwidth}
    \centering
    \includegraphics[width=2in]{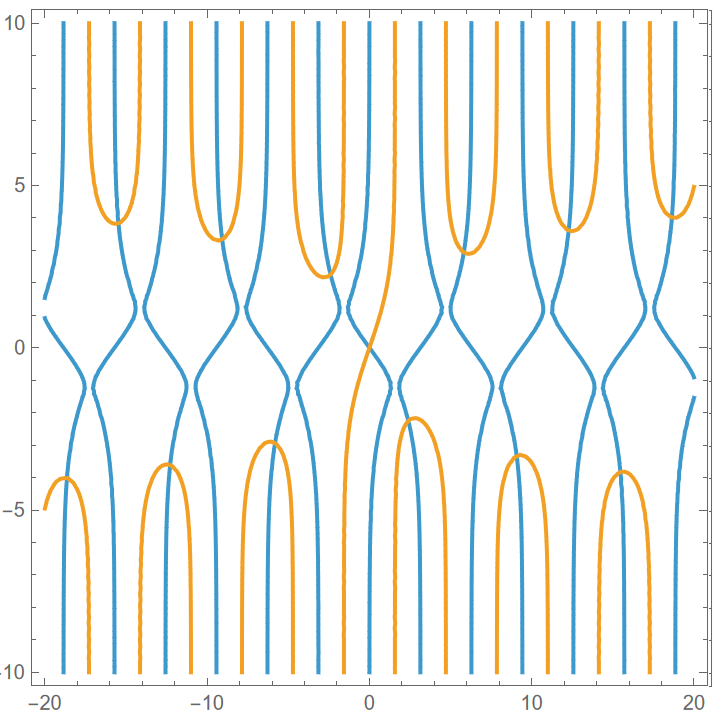}
    \caption{$\rho = 1.45$}
    \end{subfigure}%
    \begin{subfigure}{0.3\textwidth}
    \centering
    \includegraphics[width=2in]{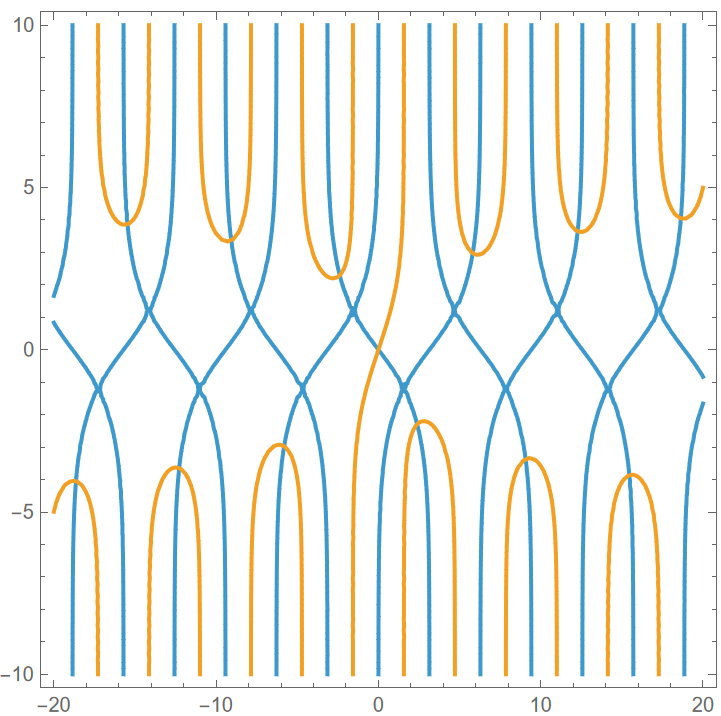}
    \caption{$\rho = \frac{1}{\alpha} \approx 1.5089$}
    \end{subfigure}
    \centering
    \begin{subfigure}{0.3\textwidth}
    \centering
    \includegraphics[width=2in]{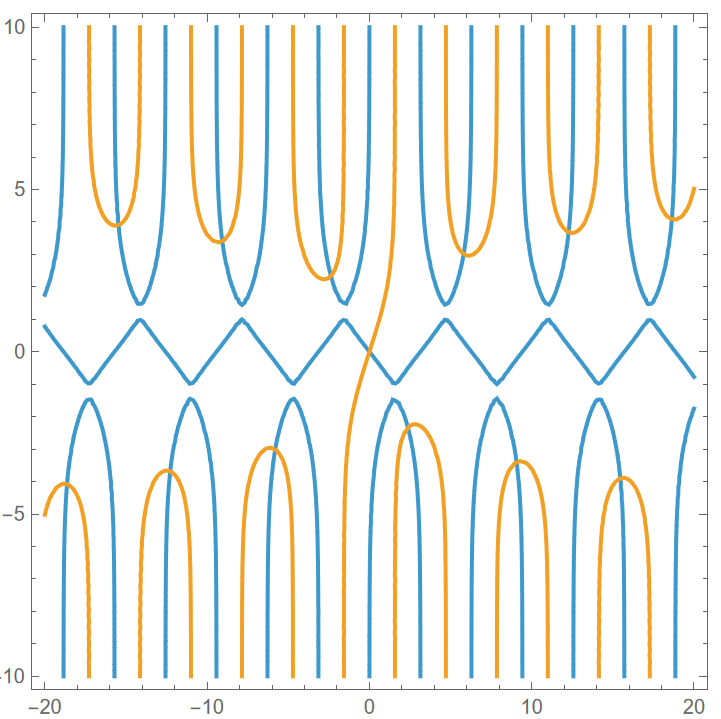}
    \caption{$\rho = 1.55$}
    \end{subfigure}%
    \caption{The bifurcation of $\Gamma_S$ (in orange) and $\Gamma_C$ (in blue) at $\rho = \frac{1}{\alpha} \approx 1.5089$.}
        \label{sinc_irho_plots}
\end{figure}
\subsection{Approximate poles of $M_\rho$}
~\\
For $\rho > 0$ and $k \in \N$, let $z_k(\rho)$ be the unique pole of $M_\rho$ (equivalently, the root of $E_\rho$) in the vertical strip $[(2k-1)\pi,2k\pi] \times \R_+$ in the complex plane, as given by Prop. \ref{erho_existence_uniqueness}. A subtle, yet important consequence of our proof is that there are no other poles of $M_\rho$ in the upper-right quadrant of $\C$. We also remark that our proof also implies $z_k(\rho)$ is the unique pole of $M_\rho$ in the entire vertical strip $[(2k-1)\pi,2k\pi] \times \R$. 

Using the evenness of $E_\rho$, and the identity $E_\rho(\overline{z}) = \overline{E_\rho(z)}$, we have $-\overline{z_k(\rho)}$ is a pole of $M_\rho$ in the upper-left plane, and so we have enumerated all of the poles of $M_\rho$ in the upper complex half-plane.

However, the poles of $M_\rho$ have no closed form expression, and so to actually compute the asymptotics, we will need to use an approximation. For $\rho > 0$ and $k \in \N$, let 
\begin{align*}
w_k(\rho) &\coloneq 2\pi k + 2i\arcsinh k \pi \rho.
\end{align*}
We will sometimes write $z_k$ and $w_k$ to clean up notation, though the $\rho$-dependence is still an essential consideration.

Now we calculate
\begin{align*}
    \sin w_k(\rho) &= i\cos 2k\pi \sinh( 2\arcsinh k \pi \rho )\\
    &= 2ik \pi \rho \sqrt{(k \pi \rho)^2 + 1}
\end{align*}
and
\begin{align*}
    4\sin^2 \frac{w_k(\rho)}{2} &= 4\left( i \cos k\pi \sinh( \arcsinh k\pi \rho) \right)^2 \\
    &= -4(k\pi\rho)^2.
\end{align*}
\noindent Therefore,
\begin{subequations}
\begin{align}
    E_\rho(w_k(\rho)) &= 4 k\pi \rho^2\left[ -\frac{\arcsinh^2 k\pi\rho}{k\pi} + 2i\arcsinh k\pi\rho\right]\label{erho_wk}\\
    E'_\rho(w_k(\rho)) &= 4 k\pi\rho\left[ \rho + i \left( \frac{\rho \arcsinh k \pi \rho}{k \pi} +  \sqrt{(k\pi\rho)^2 + 1} \right) \right]. \label{derho_w_k}
\end{align}
\end{subequations}
We also have some important bounds we will refer back to. If we assume $\rho < 1$, Lemma \ref{arcsinhu_overu} yields
\begin{subequations}
\begin{align}
    \left| E_\rho(w_k(\rho)) \right| &\leq 4 k \pi\rho^2 \left( | \arcsinh k\pi \rho| + |2i\arcsinh k \pi \rho|\right) \nonumber\\
    &\leq 4\sqrt{5}\rho (k\pi\rho)\arcsinh k\pi\rho.\label{erho_wk_upperbound}
\end{align}
Additionally, for all $\rho > 0$, we have
\begin{align}
\left| E'_\rho(w_k(\rho)) \right| &\geq 4k\pi\rho \sqrt{(k\pi\rho)^2 + 1}.    \label{derho_wk_lowerbound}
\end{align}
\end{subequations}
Now we can prove a $k$-dependent bound on the errors $|z_k(\rho) - w_k(\rho)|$ for sufficiently small $\rho$. The errors are not summable over $k$. We remark that Prop. \ref{erho_existence_uniqueness} is essential in deducing that the root of $E_\rho$ we obtain via Rouch\'e's theorem in a small disc around $w_k(\rho)$ is actually $z_k(\rho)$ and not some other root.
\begin{prop}\label{zkrk_bound_thm}
    For all $\rho > 0$ sufficiently small and for all $k \in \N$, we have
    \begin{align}
        |z_k(\rho) - w_k(\rho)| &\leq 2\sqrt{5} \frac{\rho \arcsinh k\pi\rho}{\sqrt{(k\pi\rho)^2 + 1}}. \label{zkrk_bound}
    \end{align}
\end{prop}
\begin{remark}
    Specifically, we get that \eqref{zkrk_bound} holds for all $0 < \rho < \frac{2\sqrt{5}}{55} \approx 0.0813.$
\end{remark}

\begin{proof}
First, assume $0 < \rho < 1$ so we may apply \eqref{erho_wk_upperbound} and \eqref{derho_wk_lowerbound}. Define the linear approximation $L(z) = E'_\rho(w_k)(z-w_k)$, and let $r_k = r_k(\rho) \coloneq \frac{2|E_\rho(w_k)|}{|E'_\rho(w_k)|}$.

As 
\begin{align}
    r_k \leq 2\sqrt{5} \frac{\rho \arcsinh k \pi\rho}{\sqrt{(k\pi\rho)^2 + 1}}, \label{rkbound}
\end{align}
an application of Rouch\'e's theorem will guarantee the existence of a unique root of $E_\rho$ within a distance of $r_k$ of $w_k$. Prop. \ref{erho_existence_uniqueness} guarantees this root is $z_k$ and not some other root. So to complete the proof, we must satisfy the Rouch\'e hypothesis that $|E_\rho(z) - L(z) | < |L(z)|$ for all $z$ satisfying $|z-w_k| = r_k$.

We can choose $\rho$ sufficiently small to guarantee $r_k \leq 1$ for all $k$. Indeed, from Lemma \ref{arcsinhu_overu}, we have the bound $\frac{\arcsinh u}{\sqrt{u^2 + 1}} < \frac{\arcsinh u}{u} < 1$ for all $u > 0$, and so from \eqref{rkbound}, we can derive the inequality
\begin{align*}
    |r_k| \leq 2\sqrt{5}\rho &\leq 1,
\end{align*}
which is true if we henceforth assume $\rho \leq \frac{1}{2\sqrt{5}} \approx 0.2236$.

By the Taylor theorem with integral remainder, we have
\begin{align*}
    E_\rho(z) - L(z) &= E_\rho(w_k) + \int_0^1 (1-t)(z-w_k)^2 E''_\rho(w_k + t(z-w_k))dt.
\end{align*}
Hence,
\begin{align} \label{rouchetaylor}
    |E_\rho(z) - L(z)| \leq |E_\rho(w_k)| + \frac{r_k^2}{2} \sup\limits_{|\zeta - w_k| \leq r_k} |E''_\rho(\zeta)|.
\end{align}
We remind the reader that $E''_\rho(\zeta) = 2\rho^2 + 2\cos \zeta$. Now we seek an upper bound for $M \coloneq \sup\limits_{|\zeta - w_k| \leq r_k} |E''_\rho(\zeta)|$. By working ahead, one can anticipate the need for an estimate of the form $M \leq C((k\pi\rho)^2 + 1).$
Write $\zeta \coloneq w_k + re^{i\theta}$, where $0 \leq r \leq r_k \leq 1$, and recall the inequalities $\left|\cos z\right| \leq \cosh \Imz z $ for all $z\in \C$ and $\cosh(u+1) \leq e\cosh u$ for all $u \in \R$. We therefore have
\begin{align*}
    \left| \cos \zeta \right| &\leq \cosh \Imz \zeta \\
    &\leq \cosh (2\arcsinh k\pi\rho + r) \\
    &\leq e\cosh(2\arcsinh k \pi \rho) \\
    &= 2e(k\pi\rho)^2 + e.
\end{align*}
Hence, using the assumption $\rho < 1$ and the arithmetic fact $2e + 2 <4e < 11$,
\begin{align}
    M &\leq 2\rho^2 + 2\left|\cos \zeta \right| \nonumber \\
    &\leq 2 + 2e(k\pi\rho)^2 + 2e \nonumber \\
    &\leq 11((k\pi\rho)^2 + 1) \label{Mbound}
\end{align}

We return our attention to the Rouch\'e hypothesis. Recall that $|L(z)| = |E'_\rho(w_k)|r_k$, when $|z-w_k| = r_k$. Now from \ref{rouchetaylor} and the definition of $r_k$, we have
\begin{align*}
    |E_\rho(z) - L(z)| &\leq \frac{|E'_\rho(w_k)|r_k}{2} + \frac{Mr_k^2}{2} \\
    &\leq \frac{|L(z)|}{2} + \frac{Mr_k^2}{2}.
\end{align*}
And so the Rouch\'e hypothesis will be satisfied when
\begin{align*}
    \frac{Mr_k^2}{2} &< \frac{|L(z)|}{2}
\end{align*}
which is equivalent to
\begin{align}
    2M |E_\rho(w_k)|< |E'_\rho(w_k)|^2.
\end{align}
From \eqref{erho_wk_upperbound}, \eqref{derho_wk_lowerbound}, and \eqref{Mbound}, we need to strictly chain together the inequalities
\begin{align*}
    2M |E_\rho(w_k)| &\leq 88\sqrt{5}\rho(k\pi\rho)((k\pi\rho)^2 + 1) \arcsinh k \pi \rho
\end{align*}
and
\begin{align*}
    16(k\pi\rho)^2((k\pi\rho^2)+1) &\leq |E'_\rho(w_k)|^2.
\end{align*}
Once again, we apply Lemma \ref{arcsinhu_overu}, and we may complete the chain if we assume
\begin{align*}
   88\sqrt{5}\rho < 16.
\end{align*}

\end{proof}

\subsection{Asymptotics of the approximation $\tilde{I}(\rho)$}
~\\
From the calculus of residues, we have that for all poles $z$ of $M_\rho$,
\begin{align*}
    \text{Res}(M_\rho,z) &= \frac{\rho^2 + 1}{E'_\rho(z)}.
\end{align*}
Taking into account all of the relevant symmetries, we derive the exact formula
\begin{align}
    I(\rho) &= -4\pi(\rho^2 + 1) \sum\limits_{k=1}^\infty \Imz \frac{1}{E'_\rho(z_k(\rho))} .
\end{align}
As we saw, there is no closed-form expression for the poles of $M_\rho$. However, we can define an approximate sum $\tilde{I}(\rho)$, which sums over the approximate sequence of poles.

Define the approximation to $I(\rho)$ with the series
\begin{align*}
\tilde{I}(\rho) &\coloneq -4\pi  (\rho^2 + 1) \sum\limits_{k=1}^{\infty}  \Imz \frac{1}{E'_\rho(w_k(\rho))}.
\end{align*}
\noindent Now that we have a series in closed form, we can calculate its asymptotics. 

The asymptotic analysis is driven termwise by a complex perturbation $D_k(\rho)$ of the terms of the real series $G(\rho)$, both defined below. Rather than tracking the real and imaginary components of $\frac{1}{D_k}$ separately, the subtraction trick $\frac{1}{a} - \frac{1}{b} = \frac{b-a}{ab}$ allows for the modulus of the complex perturbation to be absorbed cleanly.
\begin{lemma}\label{Itilde-asymp}
    As $\rho \to 0$,
    \begin{align*}
        \tilde{I}(\rho) &\sim \frac{\log \frac{1}{\rho}}{\rho}.
    \end{align*}
\end{lemma}

The proof of Lemma \ref{Itilde-asymp} relies on the following foundational Euler-Maclaurin calculation, whose technique can be referenced from several sources, including \cite{debruijn1981}.

For $\rho > 0$, let 
\begin{align*}
    G(\rho) &\coloneq \sum\limits_{k=1}^\infty \frac{1}{k\sqrt{(k\pi\rho)^2 + 1}}.
\end{align*}
\begin{lemma} \label{grhoasymptotics}
As $\rho \to 0$,
    \begin{align*}
        G(\rho) &\sim \log \frac{1}{\rho}.
    \end{align*}
\end{lemma}

\begin{proof}
By the limit comparison test with the convergent series $\sum\limits_{k} \frac{1}{(k\pi\rho)^2}$, we have that $G(\rho)$ is indeed a finite positive number for all $\rho > 0$.

Let $g_\rho: [1,\infty) \to \R$ be defined by
\begin{align*}
    g_\rho(t) &\coloneq \frac{1}{t\sqrt{(\pi\rho t)^2 + 1}}.
\end{align*}
Note $g_\rho$ is nonnegative and decreasing. By the integral test for series,
\begin{align}\label{g_integraltest}
    \int_1^{\infty} g_\rho(t)dt \leq G(\rho) \leq g_\rho(1) + \int_1^{\infty}g_\rho(t)dt.
\end{align}
Now with the substitution $u = \pi \rho t$, we have
\begin{align*}
    \int_1^{\infty} g_\rho(t)dt &= \int_{\pi\rho}^{\infty} \frac{du}{u\sqrt{u^2 + 1}} \\
    &= \log \frac{u}{1 + \sqrt{u^2 + 1}} \bigg\rvert^{\infty}_{\pi\rho} \\
    &= \log \frac{1 + \sqrt{(\pi\rho)^2+1}}{\pi\rho} \\
    &= \log \frac{1}{\rho} + \log \frac{1+\sqrt{(\pi\rho)^2 + 1}}{\pi}.
\end{align*}
Therefore from \ref{g_integraltest}, we have for all $\rho > 0$,
\begin{align}\label{G_asympequiv}
    \log \frac{1+\sqrt{(\pi\rho)^2 + 1}}{\pi} \leq G(\rho) - \log \frac{1}{\rho} &\leq \frac{1}{\sqrt{(\pi\rho)^2 + 1}} + \log \frac{1+\sqrt{(\pi\rho)^2 + 1}}{\pi}.
\end{align}
By taking $\rho \to 0$ on both sides of \ref{G_asympequiv}, we have that the quantity $|G(\rho) - \log\frac{1}{\rho}|$ is bounded for all $\rho$ sufficiently small. By recalling that $\log \frac{1}{\rho} \to \infty$ as $\rho \to 0$, we conclude $\lim\limits_{\rho \to 0} \frac{G(\rho)}{\log \frac{1}{\rho}} = 1$. Equivalently, $G(\rho) \sim \log \frac{1}{\rho}$ as $\rho \to 0$.

\end{proof}

We now prove Lemma \ref{Itilde-asymp}.

\begin{proof} 
\noindent Recalling the trivial property $\left|\Imz z\right| \leq |z|$, we can use \ref{derho_wk_lowerbound} to conclude
\begin{align}
    |\tilde{I}(\rho)| &\leq \left( \frac{\rho^2 +1}{\rho} \right) G(\rho). \label{Itilde_absolutebound}
\end{align}
Hence $\tilde{I}(\rho)$ is indeed a positive finite number for all $\rho > 0$. Next, let us define
\begin{align}\label{dkrho}
    D_k(\rho) &\coloneq \frac{\rho \arcsinh k \pi \rho}{k\pi} + \sqrt{(k\pi\rho)^2 + 1} - i\rho,
\end{align}
\noindent which allows us to write
\begin{align*}
    \tilde{I}(\rho) &= (\rho^2 + 1)\frac{1}{\rho}\sum\limits_{k=1}^\infty \Rez \frac{1}{kD_k(\rho)}.
\end{align*}
\noindent As $\rho^2 + 1 \to 1$ as $\rho \to 0$, it now suffices to ignore the $\frac{1}{\rho}$ factor and prove $F(\rho) \coloneq\sum\limits_{k = 1}^\infty \Rez \frac{1}{kD_k(\rho)} \sim G(\rho)$ as $\rho \to 0$.

\noindent We will compare the terms of $F$ and $G$. From Eq. \ref{dkrho}, we have
\begin{align*}
    |\Rez D_k(\rho)| &\geq \sqrt{(k\pi\rho)^2 + 1}.
\end{align*}
\noindent Similar to before, we use the trivial property $\left|\Rez z \right| \leq |z|$ to conclude,
\begin{align*}
    \left| \Rez \frac{1}{kD_k(\rho)} - \frac{1}{k\sqrt{(k\pi\rho)^2 + 1}} \right| &\leq \frac{1}{k} \left| \frac{ -\frac{\rho \arcsinh k \pi \rho}{k \pi} + i\rho}{D_k(\rho)\sqrt{(k\pi\rho)^2 + 1}} \right| \\
    &\leq \frac{1}{k((k\pi\rho)^2 + 1)} \left|\frac{\rho \arcsinh k \pi \rho}{k \pi} - i\rho \right| \\
    &\leq \frac{2\rho}{k((k\pi\rho)^2 + 1)}.
\end{align*}
Note that in the last inequality, we have used Lemma \ref{arcsinhu_overu} and we have introduced the assumption that $\rho < 1$. Summing over $k$ yields
\begin{align}\label{fgbound}
    |F(\rho) - G(\rho)| &\leq  \sum\limits_{k= 1}^\infty \frac{2\rho}{k((k\pi\rho)^2 + 1)}.
\end{align}
Another Euler-Maclaurin argument will show the righthand side of \eqref{fgbound} tends to $0$ as $\rho \to 0$, finishing the proof. Indeed, let $H(\rho) \coloneq  \sum\limits_{k= 1}^\infty \frac{2\rho}{k((k\pi\rho)^2 + 1)}$, and let $h_\rho: [1,\infty) \to \R$ be defined by $h_\rho(t) = \frac{2\rho}{t((\pi\rho t)^2 + 1)}$, which is nonnegative and decreasing. By the limit comparison test with the convergent series $\sum\limits_{k} \frac{1}{k^3}$, $H(\rho)$ is indeed a positive finite number for all $\rho > 0$. 

By the integral test for series,
\begin{align}
H(\rho) &\leq h_\rho(1) + \int_1^{\infty} h_\rho(t)dt \nonumber\\ 
&= \frac{2\rho}{(\pi\rho)^2 + 1} + \int_1^{\infty} h_\rho(t)dt \nonumber\\ 
&\leq 2\rho\left(1  + \int_1^{\infty} \frac{dt}{t((\pi \rho t)^2 + 1)} \right). \label{hbound}
\end{align}
Using a partial fraction decomposition, 
\begin{align*}
    \int_1^{\infty} \frac{dt}{t((\pi \rho t)^2 + 1)} &= \int_{1}^{\infty} \left[ \frac{1}{t} - \frac{(\pi \rho)^2t}{(\pi\rho t)^2+ 1} \right]dt \\
    &= \log \frac{1}{\pi \rho} + \frac{1}{2}\log (1 + (\pi\rho)^2).
\end{align*}
Finally, from \eqref{hbound}, we can see $H(\rho) \to 0$ as $\rho \to 0$. Thus, \eqref{fgbound} yields $F(\rho) \sim G(\rho)$ as $\rho \to 0$, and we are done.
\end{proof}

\subsection{Complex asymptotic equivalence of the infinite series}
~\\
We now come to the execution of the complex asymptotic equivalence. Only the series matters, so we can ignore the $-4\pi(\rho^2 +1)$ factors. Define
\begin{align*}
    J(\rho)& \coloneq \sum\limits_{k=1}^\infty \frac{1}{E'_\rho(z_k(\rho))} \\
    \tilde{J}(\rho) &\coloneq \sum\limits_{k=1}^\infty \frac{1}{E'_\rho(w_k(\rho))}.
\end{align*}
First, we show $\tilde{J}$ satisfies the hypotheses of the transfer theorem. Therefore, showing $J(\rho) \sim \tilde{J}(\rho)$ as $\rho \to 0$ will imply our desired Theorem \ref{irho_asymptotics}.

\begin{prop}\label{jtildetransferthm}
The function $\tilde{J}$ satisfies the condition of the transfer theorem, Theorem \ref{transferthm}. In fact, for all $\rho > 0$ and $k \in \N$,
\begin{enumerate}
    \item $\Imz \frac{1}{E'_\rho(w_k(\rho))} < 0$, and
    \item $\left| \Imz \frac{1}{E'_\rho(w_k(\rho))} \right| \geq \frac{\pi}{\sqrt{\pi^2 + 1}} \left| \frac{1}{E'_\rho(w_k(\rho))} \right|$.
\end{enumerate}  
\end{prop}

\begin{proof}
Recall Eq. \eqref{derho_w_k}, which states
\begin{align*}
 E'_\rho(w_k) &= 4k\pi \rho\left[ \rho + i \left( \frac{\rho \arcsinh k \pi \rho}{k \pi} +  \sqrt{(k\pi\rho)^2 + 1} \right) \right].
\end{align*}

    First, we have the general identity $\Imz \frac{1}{z} = \frac{- \Imz z}{|z|^2}$. Then from direct inspection of Eq. \eqref{derho_w_k}, we see $\Imz E'_\rho(w_k) > 0$ for all $\rho > 0$ and $k \in \N$. Therefore the sign condition is satisfied.

    Second, observe that
\begin{align}
    \frac{\Imz E'_\rho(w_k)}{\Rez E'_\rho(w_k)} &= \frac{4k\pi\rho\left( \frac{\rho \arcsinh k \pi \rho}{k \pi} +  \sqrt{(k\pi\rho)^2 + 1} \right)}{4k\pi\rho^2} \nonumber \\
    &= \frac{\arcsinh k\pi\rho}{k\pi} + \frac{\sqrt{(k\pi\rho)^2 + 1}}{\rho} \nonumber \\
    &\geq \frac{k\pi\rho}{\rho} \geq \pi. \label{reimratiobound}
\end{align}
Now, we can derive the general identity
\begin{align*}
    \left|\Imz \frac{1}{z} \right| &= \frac{ \left|\Imz z \right|}{\left|\Rez z\right|^2 + \left|\Imz z\right|^2}\\
    &= \frac{\left|\Imz z \right|}{\sqrt{\left|\Rez z\right|^2 + \left|\Imz z\right|^2}} \left| \frac{1}{z}\right|\\
    &= \frac{\left| \frac{\Imz z}{\Rez z}  \right|}{\sqrt{1 + \left| \frac{\Imz z}{\Rez z}\right|^2}} \left| \frac{1}{z}\right|.
\end{align*}
As the function $u \mapsto \frac{u}{\sqrt{u^2+1}}$ is increasing, we can apply \eqref{reimratiobound} to get the desired bound with $c = \frac{\pi}{\sqrt{\pi^2 + 1}}$. This satisfies the non-tangency condition.
\end{proof}

Even though the transfer theorem, as stated, only requires us to satisfy conditions with $\tilde{J}(\rho)$, we expect $J(\rho)$ to also satisfy the same conditions. However, as $z_k(\rho)$ does not have a closed form expression, it would be difficult to show $J(\rho)$ satisfies the sign and non-tangency conditions directly. As we have estimates on the errors $|z_k - w_k|$ given in Prop. \ref{zkrk_bound_thm}, we have a viable strategy to show $J(\rho)$ also satisfies the sign and non-tangency conditions. While the result of this proposition is not needed for the asymptotic equivalence of Prop. \ref{Jasymptotics}, we will need the controlling inequality \ref{etacontrol}.

Unlike the case with $\tilde{J}$, we can only assert the sign and non-tangency conditions for $\rho$ sufficiently small.

\begin{prop}\label{jtransferthm}
The function $J$ satisfies the condition of the transfer theorem, Theorem \ref{transferthm}. There exists $c > 0$ and $\rho_0 > 0$ such that for all $0 < \rho < \rho_0$ and $k \in \N$,
\begin{enumerate}
    \item $\Imz \frac{1}{E'_\rho(z_k(\rho))} < 0$, and
    \item $\left| \Imz \frac{1}{E'_\rho(z_k(\rho))} \right| \geq c\left| \frac{1}{E'_\rho(z_k(\rho))} \right|$.
\end{enumerate}  
\end{prop}

\begin{proof}
 We fix $k \in \N$, and begin with the assumption that $\rho > 0$ is sufficiently small so that the estimates on $|z_k - w_k|$ from Prop. \ref{zkrk_bound_thm} apply. Let $C_R = 2\sqrt{5}$ be the bounding constant from that proposition. The $R$ stands for Rouch\'e.

We begin with an estimate of $|E'_\rho(z_k) - E'_\rho(w_k)|$. By the fundamental theorem of calculus,
\begin{align*}
    E'_\rho(z_k) - E'_\rho(w_k) = (z_k-w_k)\int_0^1 E''_\rho(w_k + t(z_k-w_k))dt. 
\end{align*}
Let $\zeta = \zeta_k = \zeta_k(\rho) \coloneq z_k - w_k$, so we have that 
\begin{align*}
    |\zeta| &\leq C_R \frac{\rho \arcsinh k\pi\rho}{\sqrt{(k\pi\rho)^2 + 1}} \\
    &\leq C_R \rho,
\end{align*}
where we have used the fact $ \arcsinh u < \sqrt{u^2 + 1} $ for all $u>0$. Let us further assume $\rho < \frac{1}{C_R}$, which ensures $|\zeta| < 1$. This allows us to apply the same argument used to bound $M = \sup\limits_{|\zeta - w_k| \leq r_k} |E''_\rho(\zeta)|$ from Prop. \ref{zkrk_bound_thm} to see that for all $t \in [0,1]$,
\begin{align*}
    \left| E''_\rho(w_k + t\zeta)\right| &\leq 11((k\pi\rho)^2 + 1).
\end{align*}
Therefore,
\begin{subequations}
\begin{align}
    |E'_\rho(z_k) - E'_\rho(w_k)| &\leq 11|\zeta|((k\pi\rho)^2 + 1)  \nonumber\\
    &\leq 11C_R \rho \arcsinh( k\pi\rho) \sqrt{(k\pi\rho)^2 + 1}, \label{dEzkwk}
\end{align}
where we applied \eqref{zkrk_bound}.
Now apply \eqref{derho_wk_lowerbound} and Lemma \ref{arcsinhu_overu} to yield
\begin{align}
  \left| \frac{E'_\rho(z_k)- E'_\rho(w_k)}{E'_\rho(w_k)} \right| &=    \frac{|E'_\rho(z_k) - E'_\rho(w_k)|}{|E'_\rho(w_k)|}\nonumber \\
    &\leq \frac{11C_R}{4} \rho. \label{dEzkwk_asymp}
\end{align}
\end{subequations}
The estimate \eqref{dEzkwk_asymp} allows us to write
\begin{subequations}
\begin{align}
    E'_\rho(z_k) &= E'_\rho(w_k)(1 + \eta), \label{dErho1+eta}
\end{align}
for some $\eta \in \C$ such that $|\eta| \leq \frac{11C_R}{4}\rho$. In turn, we have
\begin{align}
    \frac{1}{E'_\rho(z_k)} &= \frac{1}{E_\rho'(w_k)(1+\eta)} \nonumber \\
    &= \frac{1}{E_\rho'(w_k)} \left(1 - \frac{\eta}{1+\eta} \right). \label{dErhoinverse1+eta}
\end{align}
Now let us assume $\rho$ is small enough to guarantee $|\eta| \leq \frac{1}{2}$. Then the reverse triangle inequality yields $|1 + \eta| \geq | |1| - |\eta|| = 1 - |\eta|$. Therefore we have the controlling inequality,
\begin{align}
    \left| \frac{\eta}{1 + \eta} \right| &\leq \frac{|\eta|}{1 - |\eta|} \nonumber\\
    &\leq \frac{ \frac{11}{4}C_R \rho}{1 - |\eta|} \nonumber\\
    &\leq \frac{11}{2}C_R \rho. \label{etacontrol}
\end{align}
\end{subequations}
Now see that by the sign and non-tangency conditions shown in Prop. \ref{jtildetransferthm} and the trivial estimate $|\Imz z| \leq |z|$,
\begin{align*}
    \Imz \frac{1}{E'_\rho(z_k)} &= \Imz \left[ \frac{1}{E'_\rho(w_k)} \left(1 + \frac{-\eta}{1+\eta} \right)\right] \\
    &= \Imz \frac{1}{E'_\rho(w_k)} + \Imz \left[ \frac{-\eta}{E'_\rho(w_k)(1+\eta)}\right] \\
    &\leq -\frac{\pi}{\sqrt{\pi^2 + 1}}\left| \frac{1}{E'_\rho(w_k)} \right| +  \left| \frac{\eta}{1+\eta} \right|\left| \frac{1}{E'_\rho(w_k)} \right|. \\
\end{align*}
Thus, we can guarantee $\Imz \frac{1}{E'_\rho(z_k)} < 0$ by assuming $\rho$ is small enough to guarantee $\left| \frac{\eta}{1+\eta} \right| < \frac{\pi}{\sqrt{\pi^2 + 1}}$.
To complete this proof, it now suffices to find $c_1,c_2 > 0$ such that
\begin{align*}
    \left| \Imz \frac{1}{E'_\rho(z_k)} \right| &\geq c_1 \left| \Imz \frac{1}{E'_\rho(w_k)} \right|
\end{align*}
and
\begin{align*}
    \left| \frac{1}{E'_\rho(w_k)} \right| \geq c_2 \left| \frac{1}{E'_\rho(z_k)} \right|,
\end{align*}
joining this chain by the non-tangency of $\frac{1}{E'_\rho(w_k)}$.
We have just seen
\begin{align*}
    \left| \Imz \frac{1}{E'_\rho(z_k)} \right| &= -\Imz \frac{1}{E'_\rho(z_k)} \\
    &\geq \left( \frac{\pi}{\sqrt{\pi^2 + 1}} - \left| \frac{\eta}{1+\eta} \right| \right) \left| \frac{1}{E'_\rho(w_k)} \right| \\
    &\geq \left( \frac{\pi}{\sqrt{\pi^2 + 1}} - \left| \frac{\eta}{1+\eta} \right| \right) \left| \Imz \frac{1}{E'_\rho(w_k)} \right|.
\end{align*}
Now we shrink $\rho$ further to ensure $|\frac{\eta}{1+\eta}| < \frac{1}{2}$, which allows us to take $c_1 = \frac{\pi}{\sqrt{\pi^2 + 1}} - \frac{1}{2} > 0$.
 
Furthermore, from \eqref{dErhoinverse1+eta},
\begin{align*}
\left| \frac{1}{E'_\rho(z_k)} \right| &= \left| 1 - \frac{\eta}{1+\eta} \right|\left| \frac{1}{E'_\rho(w_k)} \right|\\
&\leq \left(1 + \left|\frac{\eta}{1+\eta}\right| \right) \left| \frac{1}{E'_\rho(w_k)} \right| \\
&\leq \frac{3}{2}\left|\frac{1}{E'_\rho(w_k)} \right|,
\end{align*}
allowing us to take $c_2 = \frac{2}{3}$.
\end{proof}

\begin{prop}\label{Jasymptotics}
    As $\rho \to 0,$
    \begin{align*}
        J(\rho) \sim \tilde{J}(\rho).
    \end{align*}
 Therefore, we may rescale by $-4\pi(\rho^2 + 1)$ and apply the transfer theorem to conclude $I(\rho) \sim \frac{\log \frac{1}{\rho}}{\rho}$ as $\rho \to 0$.
\end{prop}

\begin{proof}
Let $\vep > 0$. We recall the proof of Lemma \ref{jtransferthm}, where we obtained a sequence of complex numbers $\eta_k = \eta_k(\rho)$ such that for all $\rho > 0$ sufficiently small,
\begin{align*}
    \frac{1}{E'_\rho(z_k)} &= \frac{1}{E'_\rho(w_k)}(1 + \eta_k),
\end{align*}
where $|\eta_k| \leq C_0\rho$ uniformly in $k$ and $C_0 > 0$ is an independent constant. Note that the $\eta_k$ here was written as $\frac{-\eta}{1+\eta}$ in the prior proof. Hence, we may apply the triangle inequality, Lemma \ref{jtildetransferthm}, and the trivial bound $|\Imz z| \leq |z|$ to see that
\begin{align*}
    |J(\rho) - \tilde{J}(\rho)| &\leq \sum\limits_{k=1}^{\infty} |\eta_k| \left| \frac{1}{E'_\rho(w_k)} \right| \\
    &\leq \frac{C_0 \rho\sqrt{\pi^2 + 1}}{\pi}  \sum\limits_{k=1}^\infty \left| \Imz \frac{1}{E'_\rho(w_k)} \right| \\
    &= \frac{C_0 \rho\sqrt{\pi^2 + 1}}{\pi}  \left| \Imz \sum\limits_{k=1}^{\infty} \frac{1}{E'_\rho(w_k)}\right| \\
    &\leq \frac{C_0 \rho\sqrt{\pi^2 + 1}}{\pi} |\tilde{J}(\rho)|.
\end{align*}
Now assume $\rho < \frac{\vep \pi}{C_0\sqrt{\pi^2 + 1}}$.
\end{proof}


\section*{Acknowledgements}
The author would like to thank Mike Law, Alon Lior, Gokul Nair, Tristan Ozuch, and Jacob Reznikov for their helpful conversations and insights. The author also thanks Tobias Colding for serving as sponsoring scientist for his NSF Postdoctoral Fellowship.
\\
\\
\noindent This work is partially funded by an NSF RTG grant entitled Dynamics, Probability and PDEs in Pure and Applied Mathematics, DMS-1645643. The author is also the recipient of an NSF Mathematical Sciences Postdoctoral Fellowship, DMS-2303384.

\bibliographystyle{alpha}
\bibliography{ref}
\appendix
\section{Bounding the integral over the other sides of the contour}
In this appendix, we calculate two limits regarding $\int_{\Gamma_{R}} M_\rho(z)dz$. It is clear that the contribution to the integral of the regularizing term $-\frac{1}{z^2}$ over $\Gamma_{R} / \Gamma_{R,1}$ vanishes as $R \to \infty$. Likewise, we may safely factor out the constant $\rho^2 + 1$.  So we use $M_\rho(z)$ to denote the unregularized M\"obius density
\begin{align*}
    M_\rho(z) = \frac{1}{\rho^2z^2 + 4\sin^2\frac{z}{2}}.
\end{align*}

We remark that Theorem \ref{erho_existence_uniqueness}, which describes the discrete distribution of the poles, is logically necessary to justify that the contours avoid singularities. The descriptive theorem is independent of this appendix, so there is no circular reasoning.

\begin{lemma}
    Under the definition of $\Gamma_R$ from Fig. \ref{Gamma_R}, $ \liminf\limits_{R \to \infty} \left| \int_{\Gamma_{R,2}} M_{\rho}(z)dz + \int_{\Gamma_{R,4}} M_{\rho}(z)dz \right| =  0$.
\end{lemma}
\begin{proof}
Note that $M_\rho$ is an even function in $z$. Writing $z = x + iy$, we have $M_\rho(-x+iy) = M_\rho(x -iy) = M_\rho(\overline{z}).$ Writing out the integral over $\Gamma_{R,2}$ yields
\begin{align*}
\int_{\Gamma_{R,2}} M_\rho(z)dz &= i\int_0^R M_\rho(R + it)dt. 
\end{align*}
\noindent Likewise, writing out the integral over $\Gamma_{R,4}$ yields
\begin{align*}
    \int_{\Gamma_{R,4}} M_\rho(z)dz &= -i \int_0^R M_\rho(-R + iR -it)dt \\
    &= i \int_R^0 M_\rho(-R + is)ds \\
    &= -i \int_0^R M_\rho(\overline{R + it})dt\\
    &= -\int_{\Gamma_{R,2}} M_\rho(\overline{z})dz.
\end{align*}
\noindent As $M_\rho(\overline{z}) = \overline{M_\rho(z)}$, we have 
\begin{align*}
    \left| \int_{\Gamma_{R,2}} M_{\rho}(z)dz + \int_{\Gamma_{R,4}} M_{\rho}(z)dz \right| &= \left| \int_{\Gamma_{R,2}} M_\rho(z) - M_\rho(\overline{z}) dz \right| \\
    &\leq 2 \int_{\Gamma_{R,2}} \left|  \Imz M_{\rho}(z) \right| |dz|. \\
\end{align*}

Now let $R_k = (2k + \frac{1}{2})\pi$. Note Prop. \ref{erho_existence_uniqueness} implies there are no poles of $M_\rho$ on the nonreal segments of $\Gamma_{R_k}$. Observe for $z = R_k + it \in \Gamma_{R_k,2}$, 
\begin{align*}
    M_\rho(z) &= \frac{1}{[\rho^2(R_k^2 - t^2) + 2] + 2i[tR_k + \sinh(t)]}.
\end{align*}
\noindent As $\Imz \frac{1}{a+bi} = \frac{-b}{a^2 + b^2}$, we have
\begin{align*}
   2 \left| \Imz M_\rho(z) \right| &\leq \frac{tR_k + \sinh(t)}{1 + (tR_k + \sinh(t))^2}.
\end{align*}
\noindent It now suffices to show $$\lim\limits_{R \to \infty} \int_0^{\infty} \frac{Rt + \sinh(t)}{1 + (Rt + \sinh(t))^2}dt = 0.$$

Upon splitting up the integral's domain, we can see
\begin{align*}
    \int_0^1 \frac{Rt + \sinh(t)}{1 + (Rt + \sinh(t))^2}dt &\leq \int_0^1 \frac{Rt + \sinh(1)}{1 + (Rt)^2} dt\\
    &= \frac{\log(R^2 + 1)}{2R} + \frac{\arctan(R)}{R}\sinh(1) \to 0,
\end{align*}
\noindent as $R \to \infty$. For the integral over $[1,\infty)$, we first note that for any $u > 0$, we have $\frac{u}{1+u^2} \leq \frac{1}{u}$. Next, by the AM-GM inequality, $Rt + \sinh(t) \geq 2\sqrt{Rt\sinh(t)}$. Therefore, 
\begin{align*}
    \int_1^{\infty} \frac{Rt + \sinh(t)}{1 + (Rt + \sinh(t))^2}dt &\leq \int_1^{\infty} \frac{dt}{Rt + \sinh(t)}\\
    &\leq \frac{1}{2\sqrt{R}} \int_1^{\infty} \frac{dt}{\sqrt{t\sinh(t)}} \to 0,
\end{align*}
\noindent as $R \to \infty$, as the last improper integral converges.
\end{proof}

Extending this limit inferior to a bona fide limit as $R \to \infty$ is now seen by the invariance of the Cauchy integral formula with respect to the homotopy of contours not passing through singularities. In other words, the quantity $\left| \int_{\Gamma_{R,2}} + \int_{\Gamma_{R,4}} M_\rho(z)dz \right|$ is an eventually decreasing step function of $R$, with a discrete set of jumps, corresponding to when the contour engulfs a new pole. 

\begin{lemma}
    Under the definition of $\Gamma_R$ from Fig. \ref{Gamma_R}, $ \lim\limits_{R \to \infty} \left| \int_{\Gamma_{R,3}} M_{\rho}(z)dz \right| =  0$.
\end{lemma}
\begin{proof}
The circular arc $\Gamma_{R,3}$ can be parametrized by $z = R\sqrt{2}e^{i\theta}$, where $\frac{\pi}{4} \leq \theta \leq \frac{3\pi}{4}$. We note that
\begin{align*}
    \left| \int_{\Gamma_{R,3}} M_{\rho}(z)dz \right| &\leq \frac{\pi R \sqrt{2}}{2} \sup\limits_{\theta \in [\frac{\pi}{4},\frac{3\pi}{4}]} |M_\rho(R\sqrt{2}e^{i\theta})|.
\end{align*}
\noindent Based on Equation \eqref{mrho2}, we will write
\begin{align*}
    M_\rho(z) &= \frac{1}{\rho^2z^2 + 4\sin^2(\frac{z}{2})} \\
    &= \frac{1}{\rho^2z^2 + 2 - e^{iz} - e^{-iz}}.
\end{align*}

For all $z \in \Gamma_{R,3}$, the modulus of $\rho^2z^2$ is $2R^2\rho^2$. Writing $z = R\sqrt{2}(\cos \theta + i \sin\theta),$ we have that $-e^{iz}$ has modulus $e^{-R\sqrt{2}\sin\theta}$ whilst $-e^{-iz}$ has modulus $e^{R\sqrt{2}\sin \theta}$. Note that $\frac{\sqrt{2}}{2} \leq \sin \theta \leq 1$. Without this bound, this proof would fail.

Hence, for $R$ sufficiently large (depending on $\rho$, which is fixed), we have $\sup\limits_{\theta \in [\frac{\pi}{4},\frac{3\pi}{4}]} |M_\rho(R\sqrt{2}e^{i\theta})| \leq Ce^{-R\sqrt{2}}$, and so it follows that $\left| \int_{\Gamma_{R,3}} M_{\rho}(z)dz \right| \to 0$ as $R \to \infty$. \end{proof}

\Addresses
\end{document}